\documentclass{amsart}

\usepackage{fullpage}
\usepackage{color}
\usepackage{amssymb}
\usepackage{cancel}

\title{Two multivariate quadratic transformations of elliptic hypergeometric integrals}

\author{Fokko J. van de Bult}

\newtheorem{thm}{Theorem}[section]
\newtheorem{prop}[thm]{Proposition}
\newtheorem{lemma}[thm]{Lemma}

\newtheorem{cor}[thm]{Corollary}

\begin{document}
\begin{abstract}
Eric Rains conjectured several quadratic transformations between multivariate elliptic hypergeometric functions in \cite{RainsLittlewood}, with the integrand multiplied by interpolation functions. In this article two of these conjectures are proven in the case where the interpolation functions are constant, and we obtain a third conjecture as a corollary. Two other equations for elliptic Selberg integrals with 10 parameters, two of which multiplying to $pq/t$, are given, as they are needed in the proof. The proofs consist essentially of a calculation which strings together many elliptic Dixon transformations. Some remarks are made about using Fubini in cases in which product contours do not exist.
\end{abstract}

\maketitle

In this paper we consider quadratic elliptic hypergeometric transformations. Elliptic hypergeometric functions are generalizations of (basic) hypergeometric functions, which have been studied since the beginning of this century. The elliptic hypergeometric series are defined as series in which the quotient $a_{n+1}/a_n$ of subsequent terms is an elliptic function of $n$ (as opposed to a rational function of $n$ for ordinary hypergeometric series). We are interested in the associated multivariate integrals, which are generalizations of the beta integral.

The elliptic hypergeometric functions depend on two parameters, $p$ and $q$, just as basic hypergeometric functions depend on $q$. The quadratic nature of the transformations discussed here is the fact that the $p$ and $q$ on both sides of the equation are related by a quadratic equation. We prove one equation relating a $(p,q)$-elliptic hypergeometric function with a $(p,q^2)$-elliptic hypergeometric function, and another one relating a $(p,q)$-elliptic hypergeometric function with a $(\sqrt{p}, \sqrt{q})$-elliptic hypergeometric function.  Univariate series analogues of the equations we consider are \cite[Theorem 5.1]{SpiBailey} (for our Theorem \ref{thm71}) and \cite[Theorem 4.2]{Warnaar3} (for Theorem \ref{thmq5}). Other quadratic formulas for elliptic hypergeometric series were given  in \cite{Warnaar1},  \cite{Warnaar2}, \cite{GasperSchlosser}, \cite{ChuJia1}, and \cite{ChuJia2}. 
To the best of the author's knowledge, \textit{multivariate} quadratic transformations as given here, which are not direct consequences of univariate quadratic transformation, had not been discovered even in the basic hypergeometric case.

The specific elliptic hypergeometric functions in question are elliptic Selberg integrals with ten parameters, two of which multiply to $pq/t$ (where $t$ is the parameter describing the cross terms in the elliptic Selberg integral). Elliptic Selberg integrals are an $n$-variate generalization of the elliptic beta integral \cite{Spibeta}. The six parameter Selberg integral satisfies an evaluation formula, and the eight parameter Selberg integral satisfies transformation formulas with the Weyl group of type $E_7$ as symmetry group, see \cite{Rainstrafo}. In particular, the Selberg integrals we study are slightly beyond those for which we know interesting equations.

The quadratic equations were conjectured by Eric Rains \cite[Conjectures Q3 and Q7]{RainsLittlewood}, together with 5 other quadratic equations. He managed to prove the univariate version of all of his conjectures, and most conjectures also for other small values of $n$, though not the ones under discussion in this article. His conjectures are more general in that the integrands are multiplied with certain interpolation functions, and we only prove them when these interpolation functions are constant. Unfortunately, apart from the fact that our equations imply the general conjectures for $n\leq 3$ our method of proof does not seem to be extendable to this more general case. However, as a bonus we obtain \cite[Conjectures Q1]{RainsLittlewood} in the case the interpolation function is constant, as a corollary to the second quadratic transformation we prove. 

As an intermediate step in the proofs we show two other transformation formulas for 
arbitrary Selberg integrals with ten variables, two of which multiplying to $pq/t$. The first one is the ``induction enabler'' transformation, Proposition \ref{propbluebox}, which equates an $n$-variate integral with a double integral of an $(n-1)$-variate integral and a univariate integral. This $(n-1)$-variate integral is again a Selberg integral of the same type, so this proposition lends itself perfectly for inductive arguments. 

The second transformation, Proposition \ref{propas}, relates such an $n$-dimensional Selberg integral (again ten parameters, two multiplying to $pq/t$) with the double integral of an $n$-variate integral with a univariate integral. The eight parameters of the $n$-variate integral which are not in the pair multiplying to $pq/t$ are transformed in a way very reminiscent of the Selberg transformation for the eight variable Selberg integral. Therefore we have dubbed this transformation the almost-Selberg transformation. 

All the proofs in this article essentially consist of stringing together a long list of known elliptic hypergeometric transformations, intertwined with applications of Fubini (to change the order of integration). This has more often been seen to be fruitful method of obtaining new interesting formulas. For example the univariate version of the elliptic Selberg transformation \cite{Spitheta} can most easily be proven in this way. Other examples include \cite{vdBC1} and \cite{vdBWF4}. One tricky issue in these proofs is always determining whether Fubini is actually applicable. The usual way to see this is by first choosing the parameters correctly so that all integration contours can be chosen to be the unit circle (thus constant), in which case Fubini is clearly valid, and then using analytic extension to obtain the result for general parameters. However, in this article we are unable to do so, and thus we devote an entire section on discussing the application of Fubini on integrals of meromorphic functions. This generalization of Fubini should prove useful in other applications as well.

The article is organized as follows. We start with a section on notation. 
Subsequently we discuss the two multivariate elliptic hypergeometric beta integrals, the already mentioned Selberg integral, and the Dixon integral, another generalization of the elliptic beta integral. This is followed by a section devoted to extending Fubini's theorem. Sections \ref{secbluebox} and \ref{secas} give the induction enabler transformation, respectively the almost-Selberg transformation. The final two sections are devoted to stating and proving the actual quadratic transformations.

\subsection*{Acknowledgements} I would like to thank Eric Rains for our many discussions, especially those regarding the application of Fubini.

\section{Notation}
We define the elliptic Pochhammer symbols and the elliptic gamma function as usual
\[
(x;p,q) = \prod_{r,s\geq 0} 1-p^r q^s x,\qquad 
\Gamma_{p,q}(x) = \frac{(pq/x;p,q)}{(x;p,q)}.
\]
The infinite product converges if $|p|, |q|<1$, so we will assume this throughout the paper. 
We use the standard notation for products of these terms, thus for example
\[
\Gamma_{p,q}(a_1,\ldots,a_n) = \prod_{k=1}^n \Gamma_{p,q}(a_k), \qquad 
\Gamma_{p,q}(a y^{\pm 1} z^{\pm 1}) = \Gamma_{p,q}(ayz,ay/z,az/y,a/yz)
\]
The reflection equation 
\[
\Gamma(z, pq/z)=1
\]
will be used very often to simplify terms.

Finally we use the following notation for the parameter independent part of the different integrands. 
\begin{align*}
\Delta_{I;p,q}^{(n)}(z) &:= \frac{(p;p)^n (q;q)^n}{2^n n!} \frac{1}{\prod_{1\leq j<k\leq n}\Gamma_{p,q}(z_j^{\pm 1}z_k^{\pm 1}) \prod_{j=1}^n \Gamma_{p,q}(z_j^{\pm 2})} \prod_{j=1}^n \frac{dz_j}{2\pi i z_j}, \\
\Delta_{I\!I;p,q}^{(n)}(t;z) &:=\frac{(p;p)^n (q;q)^n\Gamma_{p,q}(t)^n}{2^n n!}  
\prod_{1\leq j<k\leq n} \frac{\Gamma_{p,q}(tz_j^{\pm 1}z_k^{\pm 1})}
{\Gamma_{p,q}(z_j^{\pm 1}z_k^{\pm 1})} \prod_{j=1}^n \frac{1}{  \Gamma_{p,q}(z_j^{\pm 2})} \prod_{j=1}^n \frac{dz_j}{2\pi i z_j}.
\end{align*}
To simplify notation we will generally omit the $p,q$ from these functions. However, as quadratic transformations intrinsically imply that we also have to deal with other values for $p$ and $q$, we will use the notations
\[
\tilde \Gamma= \Gamma_{p,q^2}, \qquad \hat \Gamma = \Gamma_{\sqrt{p},\sqrt{q}},
\]
and similarly for other functions (such as $\Delta_I^{(n)}$ and $\Delta_{II}^{(n)}$).

For the part which does depend on parameters we recognize that we always have permutation symmetry, which allows us to use the abbreviations 
\[
\Gamma(x z^{\pm 1}) = \prod_{j=1}^n \Gamma(xz_j^{\pm 1}), \qquad 
\Gamma(x y^{\pm 1} z^{\pm 1}) = \prod_{j=1}^n \prod_{k=1}^m \Gamma( xy_k^{\pm 1}z_j^{\pm 1}), 
\]
if we have $n$ parameters $z_j$ and $m$ parameters $y_k$.

\section{Two types of elliptic hypergeometric beta integrals}
When restricted to a single set of integration variables, the integrals that we consider will always be of one of two forms. They will either be Dixon or Selberg integrals (in the notation of \cite{Rainstrafo} they were called of type $I$ or type $I\!I$). 

\subsection{The Dixon integral}
The Dixon integral is an integral of the form (for parameters $t_r\in \mathbb{C}^*$)
\begin{equation}\label{eqdefid}
I_D(t_r) := \int \Delta_I^{(n)}(z) \prod_{r=1}^{2k} \Gamma( t_r z^{\pm 1}),
\end{equation}
where the parameters $t_r$ satisfy the balancing condition $\prod_r t_r = (pq)^{k-n-1}$. 
The contour for this integral is the $n$-fold product of  the unit circle if all $|t_r|<1$. The only interesting instances of this integral have $k\geq n+2$, so in those cases this condition defines a non-empty open subset of the space of parameters. Thus the Dixon integral $I_D(t_r)$ is a holomorphic function of the parameters in this set. It turns out we can extend this function analytically to a meromorphic function on $(\mathbb{C}^*)^{2k}$. 

%Indeed \cite[Theorem 10.6]{Rainstrafo} shows that 
%\[
%\prod_{1\leq r<s\leq 2k} (t_rt_s;p,q) \int \Delta_I^{(n)}(z) \prod_{r=1}^{2k} \Gamma( t_r z^{\pm 1}),
%\]
%defines a holomorphic function on $(\mathbb{C}^*)^{2k}$. In particular this allows us to determine the possible locations (and degrees) of the poles of the Dixon integral. In those cases when $t_rt_s \not \in p^{\mathbb{Z}_{\leq 0}} q^{\mathbb{Z}_{\leq 0}}$ (for $1\leq r,s\leq 2k$, where $r$ and $s$ might be equal), we can define $I_D(t_r)$ as the integral \eqref{eqdefid} with a modified contour. Indeed the contour can be chosen to be the $n$-fold product of any closed contour which circles the origin once in positive direction and separates the poles at $z_i=t_r p^{\mathbb{Z}_{\geq 0}}q^{\mathbb{Z}_{\geq 0}}$ from those at $z_i=t_r^{-1}  p^{\mathbb{Z}_{\leq 0}}q^{\mathbb{Z}_{\leq 0}}$. Note that the condition $t_rt_s \not \in p^{\mathbb{Z}_{\leq 0}} q^{\mathbb{Z}_{\leq 0}}$ is equivalent to $(t_rt_s;p,q)=0$, so, as expected, we have no contours for the locations where $I_D(t_r)$ has a pole.

Whenever a Dixon integral appears in this article we will interpret it as the meromorphic extension of the integral discussed here, specialized at some values for the $t_r$. In particular we will not specify any contours. The same will be true for the Selberg integral discussed next.

\subsection{The Selberg integral}
The Selberg integral, for parameters $t_r\in \mathbb{C}^*$ satisfying the balancing condition 
$t^{2(n-1)} \prod_r t_r = (pq)^{k-2}$ is given by 
\begin{equation}\label{eqdefis}
I_S(t;t_r) := \int \Delta_{I\!I}^{(n)}(t;z) \prod_{r=1}^{2k} \Gamma(t_r z^{\pm 1}).
\end{equation}
The integration contour can once again be chosen to be the $n$-fold product of the unit circle if $|t|<1$ and $|t_r|<1$ for all $r$. This integral  can also be extended analytically, in this case to a meromorphic function for $(t,t_r) \in (\mathbb{C}^*)^{2k+1}$. 

%Indeed \cite[Theorem 10.7]{Rainstrafo} tells us that 
%\[
%\prod_{i=0}^{n-1} \prod_{1\leq r<s\leq 2k} (t^i t_rt_s;p,q) \int \Delta_{I\!I}^{(n)}(t;z) \prod_{r=1}^{2k} \Gamma(t_r z^{\pm 1})
%\]
%is a holomorphic function on this domain. As for the Dixon integral, this analytic extension can usually still be written as the integral \eqref{eqdefis}, but with a deformed contour. If $t_rt_s \not \in t^{\mathbb{Z}_{\leq 0}} p^{\mathbb{Z}_{\leq 0}} q^{\mathbb{Z}_{\leq 0}}$ (for $1\leq r,s\leq 2k$, again $r$ and $s$ might be equal) 
%we can choose $C^n$ for any closed contour $C$ circling the origin once in positive direction, which includes $tC$ and separates the poles at $z=t_r p^{\mathbb{Z}_{\geq 0}}q^{\mathbb{Z}_{\geq 0}}$ from their reciprocals.

\subsection{Known transformations}
We recall the elliptic Dixon transformation \cite[Theorem 3.1]{Rainstrafo}. Let $k=m+n+2$. Under the balancing condition $\prod_{r=1}^{2k} t_r = (pq)^{m+1}$ we have
\[
\int_{C^n} \Delta_I^{(n)}(z;p,q)  \prod_{r=1}^{2k} \Gamma_{p,q}(t_r z^{\pm 1}) 
= 
\prod_{1\leq r<s\leq 2k} \Gamma_{p,q}(t_rt_s) 
\int_{C^m} \Delta_I^{(m)}(z;p,q) \prod_{r=1}^{2k} \Gamma_{p,q}(\frac{\sqrt{pq}}{t_r} z^{\pm 1}),
\]
as an equation between meromorphic functions. The case $m=0$ results in the evaluation formula 
\[
\int_{C^n} \Delta_I^{(n)}(z;p,q) \prod_{r=1}^{2(n+2)} \Gamma_{p,q}(t_r z^{\pm 1})
= \prod_{1\leq r<s\leq 2(n+2)} \Gamma_{p,q}(t_rt_s),
\]
for parameters $t_r$ ($1\leq r\leq 2(n+2)$) satisfying the balancing condition $\prod_r t_r =pq$.
The univariate version of this evaluation is the famous elliptic beta integral evaluation \cite{Spibeta}.

Another important transformation is the elliptic Selberg transformation \cite[Theorem 9.7]{Rainstrafo}. Let $t_r\in \mathbb{C}^*$ satisfy the balancing condition $t^{2(n-1)} \prod_{r=1}^8 t_r=(pq)^2$ then 
\begin{multline}\label{eqtrafo2}
\int \Delta_{II}^{(n)}(t;z) \prod_{r=1}^8 \Gamma(t_r z^{\pm 1}) \\ =
\prod_{j=0}^{n-1} \prod_{1\leq r<s\leq 4} \Gamma(t^j t_rt_s) \prod_{5\leq r<s\leq 8} \Gamma(t^j t_rt_s)
\int \Delta_{II}^{(n)}(t;z) \prod_{r=1}^4 \Gamma(t_rsz^{\pm 1}) \prod_{r=5}^8 \Gamma(\frac{t_r}{s} z^{\pm 1}),
\end{multline}
where $s^2 = \frac{pq}{t^{n-1} t_1t_2t_3t_4}=\frac{t^{n-1}t_5t_6t_7t_8}{pq}$. We will only be using this transformation in the univariate case, which was already given in \cite{Spitheta}. 
We can iterate this transformation, in which case we obtain that the associated symmetry group of this transformation equals the Weyl group of type $E_7$.

A third transformation was first conjectured in an earlier version of \cite{RainsLittlewood} and was proven by the author in \cite{vdBC1}. Under the balancing condition $\prod_{r=1}^4 t_r=t^{2+m-n}$ we have 
\begin{multline}\label{eqtrafo3}
\int \Delta_{II}^{(n)}(t;z) \prod_{r=1}^4 \Gamma(t_rz^{\pm 1})
\prod_{r=1}^{m+n} \Gamma(\sqrt{\frac{pq}{t}} v_r z^{\pm 1}) 
\\ = 
\prod_{r=1}^4 \prod_{s=1}^{m+n} \Gamma(\sqrt{\frac{pq}{t}} t_r v_s^{\pm 1})
\prod_{i=m+1}^n \prod_{1\leq r<s\leq 4} \Gamma(t^{n-i}t_rt_s)
\int \Delta_{II}^{(m)}(t;z) \prod_{r=1}^4 \Gamma(\frac{t}{t_r} z^{\pm 1})
\prod_{r=1}^{m+n} \Gamma(\sqrt{\frac{pq}{t}} v_r z^{\pm 1}).
\end{multline}

\subsection{Notational remarks}
The proofs in this article contain long strings of identities between different integrals, which are applications of the three transformation given above and a few propositions proven below. In order to save space we use an abbreviated notation to describe why these identities hold. In particular we write 
\begin{itemize}
\item $\stackrel{var}{=}$ if we apply a Dixon transformation in the variable $var$;
\item $\stackrel{W(E_7):var}{=}$ if we apply a Selberg transformation, or the iterated version of the Selberg transformation, in the variable $var$;
%\item $\stackrel{[1]:var}{=}$ if we apply the transformation from \cite{vdBC1};
\item $\stackrel{{\color{blue} \boxtimes}:var}{=}$ if we apply the induction enabler transformation, Proposition \ref{propbluebox};
%\item $\stackrel{\cancel{W(E_7)}:var}{=}$ if we apply the almost Selberg transformation, Proposition \ref{propas}.
\end{itemize}
We will also identify $\Delta_{I\!I}^{(1)}(t;z) = \Gamma(t) \Delta_I^{(1)}(z)$ without further comments, and use balancing conditions to simplify terms where appropriate. In the calculations themselves we will interchange the order of integration as desired, but at the end of the proofs we will always quickly discuss why it was appropriate in the calculation at hand.

\section{Remarks on Fubini}\label{secFubini}
In this section we consider multiple integrals of meromorphic functions in several variables, over closed contours. When considering such a multiple integral, in principal one can allow the contour of the inner integral (say with respect to the variable $z_1$) to depend on the variable of the outer integral (say $z_2$). Since the integrand is meromorphic we have considerable freedom in choosing what the contour should be, without changing the value of the integral, in view of Cauchy's integral theorem. If we choose the inner contour such that it does not depend on $z_2$ then we can use Fubini's theorem to interchange the order of integration (note that we integrate a continuous function over a compact set, so it is surely measurable with finite absolute integral). Unfortunately, it is not always possible to choose the inner contour independent of $z_2$, as this might involve having to move the contour over some poles of the integrand. In this section we will consider this situation and show that in many cases we can still apply Fubini. The discussion here is focused on $BC$-symmetric elliptic hypergeometric integrals, but the ideas can easily be extended to a more general setting.

The integrals we consider are of the following form
\[
\int_{C^n} \frac{hol(z)}{\prod_{r,i} (t_{ri} z_i^{\pm 1};p,q) \prod_{r,i,j} (u_{rij}z_i^{\pm 1}z_j^{\pm 1};p,q)} \prod_{j=1}^n \frac{dz_j}{2\pi i z_j},
\]
where the products are over some arbitrary number of terms, and $hol(z)$ is some holomorphic function. If all $|t_{ri}|<1$ and all $|u_{rij}|<1$ we can take all contours to be the unit circle. In particular, we can apply Fubini in this case and the integral takes the same value regardless of the order of integration. As a function of the $t_{ri}$ and the $u_{rij}$ we can analytically extend the integral to a meromorphic function for $t_{ri} \in \mathbb{C}^*$ and $u_{rij} \in \mathbb{C}^*$. As an equation between meromorphic function the integrals with different orders of integration are still equal to each other, so Fubini is still valid.

However, there is a complication. In particular, we are not interested in these integrals where all variables are independent, but we want to specialize these variables in a certain way. A meromorphic function is, locally, a quotient of two holomorphic functions, and specializing it can be delicate in the case this specialization becomes $\frac{0}{0}$. Even in the case where the order of the pole equals the order of the zero in a point, it may be ill-defined. For example, specializing $\frac{w^2+z^2}{w^2-z^2}$ in $w=z=0$ gives different results depending on whether we first set $w=0$ and then $z=0$, or first $z=0$ and then $w=0$, or first $w=iz$ and then $w=0$, etc. In particular, if two functions are equal as meromorphic functions, we can only say that the specialization of these two functions are equal, if this specialization avoids the zero-divisor of the respective denominators. It thus becomes vital to explicitly write the integrals as the quotient of two holomorphic functions, so we can determine if this denominator vanishes in the specialization.

The following lemma considers the analytic properties of the integrals we want to consider. In practice we expect they have many fewer poles than what the lemma predicts, but for our purposes this lemma sufficiently restricts their locations. The reader should compare this result to the stronger, but more specialized, statements in \cite[Theorem 10.6 and 10.7]{Rainstrafo}.
\begin{lemma}\label{lemmero}
Let $F$ be defined for $z_i \in \mathbb{C}^*$ ($1\leq i\leq n$), 
$0<|p|,|q|<1$ with parameters
$t_{ri}\in \mathbb{C}^*$ (for $1\leq r\leq R_i$, we might have no $t_{ri}$ for some values of $i$) and $u_{rij} \in \mathbb{C}^*$ (for $1\leq r\leq U_{ij}$ and $1\leq i<j\leq n$, for some $i,j$ we might have $U_{ij}=0$ and we write $u_{rij}=u_{rji}$ if $i>j$). Suppose the function
\[
F(z_i;t_{ri}; u_{rij}) 
\prod_{r,i} (t_{ri} z_i^{\pm 1};p,q) \prod_{r,i<j} (u_{rij}z_i^{\pm 1}z_j^{\pm 1};p,q)
\]
is holomorphic on the domain  $0<|p|,|q|<1$, $z_i, t_{ri}, u_{rij} \in \mathbb{C}^*$. 

To these data we assign a graph $G$ with multiple edges and half-edges (i.e., edges with only one side attached to a vertex). The vertex set is $\{1,2,\ldots,n\}$, the number of edges between $i$ and $j$ is $U_{ij}$ and there are $R_i$ half-edges attached to vertex $i$. The edges are labeled with the $u_{rij}$, and the half-edges with $t_{ri}$. 

Then the function defined for $|t_{ri}|<1$ and $|u_{rij}|<1$ by 
\[
\prod_{\lambda} (w(\lambda);p,q) \prod_{\mu} (w(\mu),pw(\mu),qw(\mu),pqw(\mu);p,q)^2 \int_{C^n} F \prod_{j=1}^n  \frac{dz_j}{2\pi i z_j}
\]
where $C$ is the unit circle, extends uniquely to a holomorphic function on the domain 
$t_{ri}, u_{rij} \in \mathbb{C}^*$ and $|p|, |q|<1$. 

Here the product over $\lambda$ is over all paths in $G$  of length at most $2\cdot 3^{n-1}$ which start and end at an half-edge. The product of the labels of the (half-)edges in the path is $w(\lambda)$. Note that such a path contains exactly two half-edges, possibly the same ones. Also note that, unless the path is symmetric, we count a path twice: once in each direction. 
Here the length of the path is the total number of half-edges and edges in the path, counting multiplicity and counting half-edges as 1 (not $\frac12$).

The product over $\mu$ is over all closed paths in $G$ of length at most $4 \cdot 3^{n-2}$. Note that such a closed path can only contain edges, so $w(\mu)$ is a product of $u_{rij}$'s. We count the paths multiple times, once for each direction and starting point.

Away from the zeros of the prefactor, this function can be obtained by integrating the integrand over contours (possibly disconnected) which are deformations of the unit circle (in principle we have different contours for different integration variables, with the contours for the inner integrals depending on the value of the outer variables)\footnote{We can describe exactly the homology class these contours should be contained in, but refrain from doing so to simplify the exposition.}.
\end{lemma}
In particular we see that we can apply Fubini whenever, for each path in the graph (either closed or ending in two half-edges), the product over the labels of the edges is not contained in $p^{\mathbb{Z}_{\leq 0}} q^{\mathbb{Z}_{\leq 0}}$. The remainder of this article is filled with examples in which this is true, so to emphasize the delicacy of this issue, we now give an example in which Fubini fails. Consider, under the balancing conditions $t^2v_1v_2=1$ and $t^2u_1u_2u_3u_4=pq$ the double integral
\[
\iint \Delta_I^{(1)}(y) \Delta_I^{(1)}(z) \Gamma(ty^{\pm 1}z^{\pm 1}) \prod_{r=1}^4 \Gamma(u_rz^{\pm 1}) \prod_{r=1}^2 \Gamma(v_ry^{\pm 1})
\]
If we first integrate over $y$, the inner integral equals 0 identically, irrespective of the value of $z$ (essentially it is a special case of the elliptic beta integral where two parameters multiply to $pq$). Integrating the function 0 over the variable $z$ gives the value 0 for the double integral. On the other hand first integrating to $z$ leads (using the elliptic beta integral evaluation) to 
\[
\Gamma(t^2) \prod_{1\leq r<s\leq 4} \Gamma(u_ru_s)
\int \Delta_I^{(1)}(y) \prod_{r=1}^4 \Gamma(tu_r y^{\pm 1}) \prod_{r=1}^2 \Gamma(v_ry^{\pm 1})
= \Gamma(t^{\pm 2}) \prod_{r=1}^4 \Gamma(tu_rv_1,tu_rv_2).
\]
In particular the equality fails. The problematic path in this case is the path using the edges labeled $v_1 \to t \to t \to v_2$, which passes through the vertices $y\to z\to y$.

\begin{proof}[Proof of Lemma \ref{lemmero}]
The proof is based on iterating the result from \cite[Corollary 10.3]{Rainstrafo}, where we set all $t_r$'s equal to $u_r$'s (in loc. cit.). Let us restate that corollary in our situation. Already in this case it should be noted that we can reduce the possible number of poles significantly by assuming a balancing condition, see \cite[Lemma 10.4]{Rainstrafo}.
\begin{cor}[Corollary 10.3 from \cite{Rainstrafo}]
Let $\Delta(z;t_r;p,q)$ be a function such that
\[
\Delta(z;t_r;p,q) \prod_{r} (t_r z^{\pm 1};p,q) 
\]
is holomorphic on $z,t_r\in \mathbb{C}^*$ and $0<|p|,|q|<1$. Then 
\[
G(t_r) = \prod_{r,s} (t_rt_s;p,q) \int_{C} \Delta(z;t_r;p,q) \frac{dz}{2\pi i z}
\]
extends uniquely to a holomorphic function on $t_r\in \mathbb{C}^*$, $0<|p|,|q|<1$. (In the product $\prod_{r,s}$ we include the term $(t_1^2;p,q)$ and include $(t_1t_2;p,q)$ twice, once for $r=1$, $s=2$ and once for $r=2$, $s=1$). Away from the zeros of the prefactor this extension can be obtained by integrating the integrand over any contour which contains the points $p^{\mathbb{Z}_{\geq 0}}q^{\mathbb{Z}_{\geq 0}} t_r$ and excludes their reciprocals.
\end{cor}
In particular this corresponds with the fact that in the univariate case our graph has one vertex, with several half-edges pointing out from it. In particular there are no closed paths (as there are no edges), and any path starting and ending at an half-edge is just the combination of two arbitrary half-edges.

Now we prove our statement with induction to $n$. Suppose the conditions of the theorem. 
We first consider the analyticity of the integral of just the $n-1$ variables $z_1$ until $z_{n-1}$. The corresponding graph is obtained by removing the $n$'th vertex, while transforming the edges to this $n$'th vertex into pairs of half-edges: the edge from $i$ to $n$ with label $u$ is changed into two half-edges attached to $i$ with labels $uz_n$ and $u/z_n$. 

Notice that the $z_n$-dependent part of the prefactor making the integral to $n-1$ variables holomorphic, apart from the terms $(t_{rn} z_n^{\pm 1};p,q)$ is contained in the product over paths starting and ending at an half-edge. In particular paths $\lambda$ starting at an half-edge labeled $uz_n$ and ending at an half edge $t$ give (using $z_n\to 1/z_n$ symmetry) a term in the prefactor of $(u t v(\lambda) z_n^{\pm 1};p,q)$ (where $v(\lambda)$ denotes the product of the labels of the edges of the path (excluding the half-edges at the end)), while paths $\lambda$ starting at an half-edge labeled $u z_n$ and ending at one labeled $u' z_n$ give a term in the prefactor of the form $(uu' v(\lambda) z_n^{\pm 2};p,q)$, and the remaining terms are $z_n$-independent. 
The latter can be expressed as 
\[
(uu'v(\lambda) z_n^{\pm 2};p,q) = 
(\pm \sqrt{uu'v(\lambda) } z_n^{\pm 1}, \pm \sqrt{puu'v(\lambda) } z_n^{\pm 1},\pm \sqrt{quu'v(\lambda) } z_n^{\pm 1},\pm \sqrt{pquu'v(\lambda) } z_n^{\pm 1};p,q).
\]
In particular we see that the $(n-1)$-fold integral is of the form where we can use the corollary to obtain the analytic properties of its integral to $z_n$.

The $z_n$-independent prefactors come from paths between half-edges whose labels are both independent of $z_n$, paths between half-edges labeled by $uz_n$, respectively $u'/z_n$, and closed paths within the graph of $n-1$ vertices. The first and last of these three cases carry over directly to the same paths in the graph with $n$ vertices, so those prefactors are indeed accounted for in the general lemma for $n$-variables. The middle case becomes a closed path in the graph of $n$-vertices starting at vertex $n$, and is thus also accounted for in the prefactor for the $n$-variable case.

Let us now consider what the prefactor is that we must add in consideration of the $z_n$-dependent parts of the prefactor after $n-1$ integrals. As described these prefactors fall in three general groups. We can actually combine the group consisting of terms $(t_{rn} z_n^{\pm 1};p,q)$ with the group consisting of terms $(tu v(\lambda)z_n^{\pm 1};p,q)$ for a path $\lambda$ starting at a half-edge labeled $t$ and ending at an half-edge labeled $uz_n^{\pm 1}$, as both being related to paths starting at an half-edge and ending at the vertex $n$ in the new graph. Their products will thus be related to paths starting and ending at the associated half-edges and passing through vertex $n$. The length of these paths is bounded by $2 \cdot ( 2\cdot 3^{n-2})< 2\cdot 3^{n-1}$. 

For the products of terms associated to a path $\mu$ starting at $uz_n$ and ending at $u'z_n$ and a path $\lambda$ starting at $t$ and ending at $u'' z_n$, we get 
\begin{multline*}
(\pm \sqrt{uu'v(\mu) } tu'' v(\lambda) , \pm \sqrt{puu'v(\mu) }tu'' v(\lambda) ,\pm \sqrt{quu'v(\mu) } tu'' v(\lambda) ,\pm \sqrt{pquu'v(\mu) } tu'' v(\lambda) ;p,q)
\\ =
(uu'v(\mu) t^2 (u'')^2 v(\lambda)^2;p,q).
\end{multline*}
Thus this corresponds to the path, starting and ending at the half-edge labeled by $t$, which first goes to vertex $n$ along $\lambda$ (with last step $u''$), then follows the path $\mu$ (starting with $u$ and ending with $u'$), and then returns along $\lambda$ to the original half-edge. Thus it corresponds to a path of length at most $3 \cdot (2\cdot 3^{n-2}) = 2 \cdot 3^{n-1}$. 

Finally we have to consider the case of products of terms associated to a path $\mu$ starting at $u z_n$, ending at $u'z_n$ and a path $\lambda$ starting at $\hat u z_n$ and ending at $\hat u' z_n$. This will correspond to a product of $8^2=64$ terms, which simplifies to 
\[
(uu'v(\mu) \hat u \hat u' v(\lambda),
puu'v(\mu) \hat u \hat u' v(\lambda),
quu'v(\mu) \hat u \hat u' v(\lambda), pq uu'v(\mu) \hat u \hat u' v(\lambda);p,q)^2
\]
In particular this corresponds to a closed path in the new graph starting at vertex $n$, using $\mu$ to return to $n$, and following it by $\lambda$. The total length is thus at most $2 (2\cdot 3^{n-2}) = 4\cdot 3^{n-2}$.

The final question we have to consider is if we find the same path multiple times (in which case we would have to take a certain power of the associated prefactor). However, given any path we can determine how it was created. If the path is a loop, then it has to start at vertex $n$, or the associated prefactor was already obtained in the prefactor of the $n-1$-dimensional integral. Then, the loop either returns to $n$ in the middle of the path or it does not. In the second case, the term must be obtained from combining two paths starting and ending at a half-edge of the form $uz_n$, and the break between those two paths occurs at the first return to vertex $n$, while in the first case the term is already contained in the prefactor of the $n-1$-dimensional integral associated to a path starting at $uz_n$ and ending at $u'/z_n$. For the paths between half-edges, we see that either they pass through vertex $n$ or they don't. If they don't then they are already contained in the prefactor of the $n-1$-dimensional integral. If they do, then they are the combination of two paths starting at a half edge and ending at $uz_n$, the break between them occurring exactly at the point the path reaches vertex $n$.

Note that many paths are actually impossible, so we could significantly simplify the prefactor. However doing this would make the expression more complicated and was unnecessary for our applications.
\end{proof}

\section{The induction enabler transformation}\label{secbluebox}
The first transformation we consider is the so-called induction enabler. It transforms an $n$-dimensional Selberg integral into a double integral of an $(n-1)$-dimensional Selberg integral and a one-dimensional integral. This transformation is therefore very convenient to be able to apply inductive arguments.

\begin{prop}\label{propbluebox}
For parameters $t_r, v\in \mathbb{C}^*$ satisfying the balancing condition $t^{2n-3} \prod_{r=1}^8 t_r =(pq)^{2}$ we have 
\begin{align*}
\int & \Delta_{II}^{(n)}(t;z) \prod_{r=1}^8 \Gamma(t_r z^{\pm 1}) \Gamma(\sqrt{\frac{pq}{t}} v^{\pm 1} z^{\pm 1})
 = \Gamma( t^n)  \prod_{k=1}^{n-1} \prod_{1\leq r<s\leq 8} \Gamma(t^{k-1} t_rt_s) 
  \\& \qquad \times 
\iint \Delta_{II}^{(n-1)}(t;z) \Delta_{I}^{(1)}(y) \Gamma(\sqrt{\frac{pq}{t}} z^{\pm 1}y^{\pm 1}) 
\prod_{r=1}^8 \Gamma(\frac{\sqrt{pq}}{t_rt^{\frac12n-1}} z^{\pm 1},   t_rt^{\frac12(n-1)} y^{\pm 1})
 \Gamma(\sqrt{\frac{pq}{t^n}}  v^{\pm 1} y^{\pm 1})  
 \end{align*}
\end{prop}
We want to view this transformation as an equation between meromorphic functions. In the case 
$|pq|<|t^n| <1$ and $|\sqrt{pq} t^{1-\frac12 n}|<|t_r|<1$ we get an equation between two multivariate integrals whose integration contours are products of the unit circle.

\begin{proof}
Throughout the proof we use the definition 
\begin{multline*}
F_{k,l}(t_1,\ldots,t_7;u;v;a) :=
\iiint \Delta_I^{(1)}(z) \Delta_{II}^{(k)}(t;y)\Delta_{II}^{(l)}(t;x) \Gamma(\sqrt{\frac{pq}{t}} z^{\pm 1}y^{\pm 1}, \sqrt{\frac{pq}{t}} y^{\pm 1} x^{\pm 1}) 
\\ \times \prod_{r=1}^7 \Gamma(\frac{\sqrt{pqt} }{t_r} y^{\pm 1},t_r x^{\pm 1})
\Gamma(a v^{\pm 1} z^{\pm 1})
\Gamma(u y^{\pm 1}, \frac{\sqrt{pq}}{u t^{k+l-3/2}} x^{\pm 1}),
\end{multline*}
assuming the balancing conditions 
\[
pq a^2 = t^k, \qquad 
\prod_r t_r = u (pq)^{3/2} t^{2k+1/2-l}.
\]
Of course these equations allow us to calculate $a$ directly, but the formulas are prettier if we use $a$. We will mostly ignore the $a$ in the list of variables of $F_{k,l}$ from now on though. 

The proof is ordered almost like a musical piece. It starts with an introduction, which relates the left hand side of the proposition to an $F_{3,n-2}$. Subsequently we repeat a sequence which equates an $F_{k,l}$ with an $F_{k+1,l-1}$ over and over, to arrive at an $F_{n+1,0}$. Finally we have a coda which simplifies this $F_{n+1,0}$ to the desired right hand side of the proposition.

The introduction is given by the following lemma.
\begin{lemma}\label{lem41}
We have 
\begin{multline*}
\int \Delta_{II}^{(n)}(t;z)\prod_{r=1}^8 \Gamma(t_r z^{\pm 1}) \Gamma(\sqrt{\frac{pq}{t}} v^{\pm 1} z^{\pm 1}) 
= \Gamma(\sqrt{pq} t_8 t^{n-5/2} v^{\pm 1}, \frac{pq}{t^2},t^n, t^{n-1},\frac{pq}{t}, \frac{(pq)^2}{t^3})  \\ \times  \prod_{1\leq r<s\leq 7}  \Gamma(t_rt_s, tt_rt_s) \prod_{r=1}^7 \Gamma( \sqrt{\frac{pq}{t}} t_r v^{\pm 1}, t_rt_8 t^{n-1},  t_rt_8t^{n-2})  F_{3,n-2}(tt_1, \ldots, tt_7;\frac{\sqrt{pq}}{t_8 t^{n-3/2}};v)
\end{multline*}
\end{lemma}
\begin{proof}
A direct calculation shows 
\begin{align*}
\int & \Delta_{II}^{(n)}(t;z) \prod_{r=1}^8 \Gamma(t_r z^{\pm 1}) \Gamma(\sqrt{\frac{pq}{t}} v^{\pm 1} z^{\pm 1})
\\& \stackrel{y}{=}
\Gamma(t^n) \iint \Delta_I^{(n)}(z) \Delta_I^{(n-1)}(y) \Gamma(\sqrt{t} y^{\pm 1} z^{\pm 1}) 
\prod_{r=1}^7 \Gamma(t_rz^{\pm 1}) \Gamma(\sqrt{\frac{pq}{t}} v^{\pm 1} z^{\pm 1})
%\\ & \qquad \times 
\Gamma(\frac{t_8}{\sqrt{t}} y^{\pm 1}, \frac{pq}{t_8t^{n-1/2}} y^{\pm 1}, t_8 t^{n-1}z^{\pm 1})
\\ & \stackrel{z}{=}
 \Gamma(t^n,\frac{pq}{t}, \sqrt{\frac{pq}{t}} t_8 t^{n-1} v^{\pm 1}) \prod_{1\leq r<s\leq 7} \Gamma(t_rt_s)  \prod_{r=1}^7 \Gamma( \sqrt{\frac{pq}{t}} t_r v^{\pm 1}, t_rt_8 t^{n-1})  
\\ & \qquad \times \iint \Delta_I^{(2)}(z) \Delta_{II}^{(n-1)}(t;y) \Gamma(\sqrt{\frac{pq}{t}} y^{\pm 1} z^{\pm 1}) 
\prod_{r=1}^7 \Gamma(\frac{\sqrt{pq}}{t_r}z^{\pm 1}, \sqrt{t} t_r y^{\pm 1}) \Gamma(\sqrt{t} v^{\pm 1} z^{\pm 1})
\Gamma(\frac{t_8}{\sqrt{t}} y^{\pm 1}, \frac{\sqrt{pq}}{ t_8 t^{n-1}}z^{\pm 1})
\\ & \stackrel{x}{=}
 \Gamma(t^n,t^{n-1},\frac{pq}{t}, \sqrt{\frac{pq}{t}} t_8 t^{n-1} v^{\pm 1}) \prod_{1\leq r<s\leq 7} \Gamma(t_rt_s)  \prod_{r=1}^7 \Gamma( \sqrt{\frac{pq}{t}} t_r v^{\pm 1}, t_rt_8 t^{n-1})  
\\ & \qquad  \times \iiint \Delta_I^{(2)}(z) \Delta_{I}^{(n-1)}(y) \Delta_I^{(n-2)}(x) 
\Gamma(\sqrt{\frac{pq}{t}} y^{\pm 1} z^{\pm 1}, \sqrt{t} y^{\pm 1}x^{\pm 1}) 
\\ & \qquad \qquad \times \prod_{r=1}^7 \Gamma(\frac{\sqrt{pq}}{t_r}z^{\pm 1}, \sqrt{t} t_r y^{\pm 1}) \Gamma(\sqrt{t} v^{\pm 1} z^{\pm 1})
\Gamma(\frac{\sqrt{pq}}{ t_8 t^{n-1}}z^{\pm 1}, t_8 t^{n-5/2} y^{\pm 1}, \frac{t_8}{t} x^{\pm 1}, \frac{pq}{t_8 t^{n-2}} x^{\pm 1} )
\end{align*}
Continuing with 
\begin{align*}
 & \stackrel{y}{=}
 \Gamma(t^n,t^{n-1},\frac{pq}{t}, \sqrt{\frac{pq}{t}} t_8 t^{n-1} v^{\pm 1}) \prod_{1\leq r<s\leq 7} \Gamma(t_rt_s, tt_rt_s)  \prod_{r=1}^7 \Gamma( \sqrt{\frac{pq}{t}} t_r v^{\pm 1}, t_rt_8 t^{n-1},  t_rt_8t^{n-2})  
\\ & \qquad  \times \iiint \Delta_{II}^{(2)}(\frac{pq}{t};z) \Delta_{I}^{(3)}(y) \Delta_{II}^{(n-2)}(t;x) 
\Gamma(\sqrt{t} y^{\pm 1} z^{\pm 1}, \sqrt{\frac{pq}{t}} y^{\pm 1}x^{\pm 1}) 
\\ & \qquad \qquad \times \prod_{r=1}^7 \Gamma(\sqrt{\frac{pq}{t}} \frac{1}{t_r} y^{\pm 1}, tt_r x^{\pm 1}) \Gamma(\sqrt{t} v^{\pm 1} z^{\pm 1})
\Gamma(\frac{\sqrt{pq}}{ t_8 t^{n-1}}z^{\pm 1}, \sqrt{pq} t_8 t^{n-3} z^{\pm 1}, \frac{\sqrt{pq}}{t_8 t^{n-5/2}} y^{\pm 1}, \frac{t_8}{t} x^{\pm 1})
\end{align*}
Now we can use the transformation \eqref{eqtrafo3} on the $z$-integral, where the four special parameters are $\sqrt{t}v$, $\sqrt{t}/v$, $\frac{\sqrt{pq}}{t_8t^{n-1}}$, and $\sqrt{pq} t_8 t^{n-3}$. In this case the transformation is dimension lowering, and in the final equation we lose the two parameters involving $t_8$:
\begin{align*}
& = \Gamma(\sqrt{pq} t_8 t^{n-5/2} v^{\pm 1}, \frac{pq}{t^2},t^n, t^{n-1}) \prod_{1\leq r<s\leq 7} \Gamma(t_rt_s, tt_rt_s)  \prod_{r=1}^7 \Gamma( \sqrt{\frac{pq}{t}} t_r v^{\pm 1}, t_rt_8 t^{n-1},  t_rt_8t^{n-2})  
\\ & \qquad  \times \iiint \Delta_{II}^{(1)}(\frac{pq}{t};z) \Delta_{I}^{(3)}(y) \Delta_{II}^{(n-2)}(t;x) 
\Gamma(\sqrt{t} y^{\pm 1} z^{\pm 1}, \sqrt{\frac{pq}{t}} y^{\pm 1}x^{\pm 1}) 
\\ & \qquad \qquad \times \prod_{r=1}^7 \Gamma(\sqrt{\frac{pq}{t}} \frac{1}{t_r} y^{\pm 1}, tt_r x^{\pm 1}) \Gamma(\frac{pq}{t\sqrt{t}} v^{\pm 1} z^{\pm 1}, t v^{\pm 1} y^{\pm 1})
\Gamma( \frac{\sqrt{pq}}{t_8 t^{n-3/2}} y^{\pm 1}, \frac{t_8}{t} x^{\pm 1})
\\ & \stackrel{z}{=}
 \Gamma(\sqrt{pq} t_8 t^{n-5/2} v^{\pm 1}, \frac{pq}{t^2},t^n, t^{n-1}, \frac{pq}{t}, \frac{(pq)^2}{t^3}) \prod_{1\leq r<s\leq 7} \Gamma(t_rt_s, tt_rt_s)  \prod_{r=1}^7 \Gamma( \sqrt{\frac{pq}{t}} t_r v^{\pm 1}, t_rt_8 t^{n-1},  t_rt_8t^{n-2})  
\\ & \qquad  \times \iiint \Delta_{I}^{(1)}(z) \Delta_{II}^{(3)}(t;y) \Delta_{II}^{(n-2)}(t;x) 
\Gamma(\sqrt{\frac{pq}{t}} y^{\pm 1} z^{\pm 1}, \sqrt{\frac{pq}{t}} y^{\pm 1}x^{\pm 1}) 
\\ & \qquad \qquad \times \prod_{r=1}^7 \Gamma(\sqrt{\frac{pq}{t}} \frac{1}{t_r} y^{\pm 1}, tt_r x^{\pm 1}) 
\Gamma(\sqrt{\frac{t^3}{pq}} v^{\pm 1} z^{\pm 1})
\Gamma( \frac{\sqrt{pq}}{t_8 t^{n-3/2}} y^{\pm 1}, \frac{t_8}{t} x^{\pm 1})
\end{align*}
The remaining integral is indeed an $F_{3,n-2}$ as desired.
\end{proof}

The main iterative sequence is given by 
\begin{lemma}\label{lem43}
The following transformation holds
\begin{multline*}
F_{k,l}(t_r;u;v;a) = \Gamma(t^l, \frac{pq}{t^k},\frac{t^k}{pq} , t^{\frac12 (k-1)}u v^{\pm 1}, \frac{pq}{ut^{\frac12(k+1)}} v^{\pm 1}, \frac{(pq)^{2}}{t^{k+1}})
\prod_{1\leq r<s\leq 7} \Gamma(t_rt_s) 
\\ \times \prod_{r=1}^7 \Gamma(\frac{\sqrt{pq} t_r}{u t^{k-1/2}}) F_{k+1,l-1}(t_r\sqrt{t};u \sqrt{t};v;a\sqrt{t}).
\end{multline*} 
\end{lemma}
\begin{proof}
First we introduce $w$ parameters to obtain
\begin{align*}
F_{k,l}(t_r;u;v) & \stackrel{w}{=} 
\Gamma(t^l) \iiiint \Delta_I^{(1)}(z) \Delta_{II}^{(k)}(t;y)\Delta_{I}^{(l)}(x) 
\Delta_I^{(l-1)}(w)
\Gamma(\sqrt{\frac{pq}{t}} y^{\pm 1}z^{\pm 1}, \sqrt{\frac{pq}{t}} y^{\pm 1} x^{\pm 1}, \sqrt{t} x^{\pm 1}w^{\pm 1}) 
\\ & \qquad  \times \prod_{r=1}^7 \Gamma(\frac{\sqrt{pqt} }{t_r} y^{\pm 1},t_r x^{\pm 1})
\Gamma(a  v^{\pm 1} z^{\pm 1})
\Gamma(u y^{\pm 1},  \frac{\sqrt{pq}}{u t^{k-1/2}} x^{\pm 1},\sqrt{pq} u t^{k-1} w^{\pm 1} ,\frac{\sqrt{pq}}{u t^{k+l-1}} w^{\pm 1})
\end{align*}
Continuing with 
\begin{align*}
 & \stackrel{x}{=}\Gamma(t^l) \prod_{1\leq r<s\leq 7} \Gamma(t_rt_s) 
\prod_{r=1}^7 \Gamma(\frac{\sqrt{pq} t_r}{u t^{k-1/2}}) 
\\ & \qquad \times 
\iiiint \Delta_I^{(1)}(z) \Delta_{I}^{(k)}(y)\Delta_{I}^{(k+1)}(x) 
\Delta_{II}^{(l-1)}(t;w)
\Gamma(\sqrt{\frac{pq}{t}} y^{\pm 1}z^{\pm 1}, \sqrt{t} y^{\pm 1} x^{\pm 1}, \sqrt{\frac{pq}{t}} x^{\pm 1}w^{\pm 1}) 
\\ & \qquad  \times \prod_{r=1}^7 \Gamma(t_r x^{\pm 1}, \sqrt{t} t_r w^{\pm 1})
\Gamma(a  v^{\pm 1} z^{\pm 1})
\Gamma(u y^{\pm 1}, \frac{pq}{ u t^k} y^{\pm 1},  u t^{k-1/2} x^{\pm 1},\frac{\sqrt{pq}}{u t^{k+l-1}} w^{\pm 1})
%\end{align*}
%Continuing with
%\begin{align*}
\\ & \stackrel{y}{=} \Gamma(t^l, \frac{pq}{t^k}, \frac{pq}{t}) \prod_{1\leq r<s\leq 7} \Gamma(t_rt_s) 
\prod_{r=1}^7 \Gamma(\frac{\sqrt{pq} t_r}{u t^{k-1/2}}) 
\\ & \qquad \times \iiiint \Delta_I^{(1)}(z) \Delta_{I}^{(1)}(y)\Delta_{II}^{(k+1)}(t;x) 
\Delta_{II}^{(l-1)}(t;w)
\Gamma(\sqrt{t} y^{\pm 1}z^{\pm 1}, \sqrt{\frac{pq}{t}} y^{\pm 1} x^{\pm 1}, \sqrt{\frac{pq}{t}} x^{\pm 1}w^{\pm 1}) 
\\ & \qquad  \times \prod_{r=1}^7 \Gamma(t_r x^{\pm 1}, \sqrt{t} t_r w^{\pm 1})
\Gamma(a  v^{\pm 1} z^{\pm 1})
\\ & \qquad \times \Gamma(\sqrt{\frac{pq}{t}} u z^{\pm 1}, \frac{(pq)^{3/2}}{u t^{k+1/2}} z^{\pm 1}, \frac{\sqrt{pq}}{u} y^{\pm 1}, 
\frac{ut^k}{\sqrt{pq}} y^{\pm 1}, u\sqrt{t} x^{\pm 1},\frac{\sqrt{pq}}{u t^{k+l-1}} w^{\pm 1})
\\ & \stackrel{z}{=} 
 \Gamma(t^l, \frac{pq}{t^k},a^2 , \sqrt{\frac{pq}{t}} au v^{\pm 1}, \frac{(pq)^{3/2}a}{ut^{k+1/2}} v^{\pm 1}, \frac{(pq)^{2}}{t^{k+1}})
\prod_{1\leq r<s\leq 7} \Gamma(t_rt_s) 
\prod_{r=1}^7 \Gamma(\frac{\sqrt{pq} t_r}{u t^{k-1/2}}) 
\\ & \qquad \times \iiint  \Delta_I^{(1)}(y)\Delta_{II}^{(k+1)}(t;x) 
\Delta_{II}^{(l-1)}(t;w)
\Gamma(\sqrt{\frac{pq}{t}} y^{\pm 1} x^{\pm 1}, \sqrt{\frac{pq}{t}} x^{\pm 1}w^{\pm 1}) 
\\ & \qquad  \times \prod_{r=1}^7 \Gamma(t_r x^{\pm 1}, \sqrt{t} t_r w^{\pm 1})
\Gamma(a\sqrt{t}  v^{\pm 1} y^{\pm 1}) \Gamma(u\sqrt{t} x^{\pm 1},\frac{\sqrt{pq}}{u t^{k+l-1}} w^{\pm 1}).
\end{align*}
The final integral can now be recognized as $F_{k+1,l-1}(t_r \sqrt{t}; u\sqrt{t};v)$.
\end{proof}
Finally we have the coda, which is
\begin{lemma}
We have
\begin{align*}
F_{n+1,0}(t_1,\ldots,t_7;u;v) &= \Gamma(\frac{t^{n+1}}{pq}, t^n, t^{\frac12 n} u v^{\pm 1}) \prod_{1\leq r<s\leq 7} \Gamma(\frac{pqt}{t_rt_s}) \prod_{r=1}^7 
\Gamma(\frac{\sqrt{pqt^{n+1}}}{t_r} v^{\pm 1}, \frac{\sqrt{pqt}u}{t_r}) \\& \qquad \times 
\iint \Delta_{II}^{(n-1)}(t;z) \Delta_{I}^{(1)}(y) \Gamma(\sqrt{\frac{pq}{t}} z^{\pm 1}y^{\pm 1}) 
\prod_{r=1}^7 \Gamma(\frac{\sqrt{pq}t}{t_r} z^{\pm 1},   \frac{t_r}{\sqrt{t} } y^{\pm 1})
\\ & \qquad \times \Gamma(\sqrt{\frac{pq}{t^n}}  v^{\pm 1} y^{\pm 1}) \Gamma(u\sqrt{t} z^{\pm 1}, \frac{\sqrt{pq}}{u} y^{\pm 1})
\end{align*}
\end{lemma}
\begin{proof}
Notice that the final $x$ integral in $F_{n+1,0}$ just disappears (becomes 1) as it is $0$-dimensional, so we get 
\begin{align*}
F_{n+1,0}(t_1,\ldots,t_7;u;v;a) &=
\iint \Delta_I^{(1)}(z) \Delta_{II}^{(n+1)}(t;y) \Gamma(\sqrt{\frac{pq}{t}} z^{\pm 1}y^{\pm 1}) 
\prod_{r=1}^7 \Gamma(\frac{\sqrt{pqt} }{t_r} y^{\pm 1})
\Gamma(a v^{\pm 1} z^{\pm 1}) \Gamma(u y^{\pm 1})
\\ & \stackrel{z}{=} \Gamma(a^2)
\iint \Delta_I^{(n-1)}(z) \Delta_{I}^{(n+1)}(y) \Gamma(\sqrt{t} z^{\pm 1}y^{\pm 1}) 
\prod_{r=1}^7 \Gamma(\frac{\sqrt{pqt} }{t_r} y^{\pm 1})
\\ & \qquad \times \Gamma(\frac{\sqrt{pq}}{a} v^{\pm 1} z^{\pm 1}, \sqrt{\frac{pq}{t}} a v^{\pm 1} y^{\pm 1}) \Gamma(u y^{\pm 1})
\end{align*}
Continuing with
\begin{align*}
 & \stackrel{y}{=} \Gamma(a^2, \frac{pqa^2}{t}, \sqrt{\frac{pq}{t}} au v^{\pm 1}) \prod_{1\leq r<s\leq 7} \Gamma(\frac{pqt}{t_rt_s}) \prod_{r=1}^7 
\Gamma(\frac{pqa}{t_r} v^{\pm 1}, \frac{\sqrt{pqt}u}{t_r}) \\& \qquad \times 
\iint \Delta_{II}^{(n-1)}(t;z) \Delta_{I}^{(1)}(y) \Gamma(\sqrt{\frac{pq}{t}} z^{\pm 1}y^{\pm 1}) 
\prod_{r=1}^7 \Gamma(\frac{\sqrt{pq}t}{t_r} z^{\pm 1},   \frac{t_r}{\sqrt{t} } y^{\pm 1})
\\ & \qquad \times \Gamma(\frac{\sqrt{t}}{a} v^{\pm 1} y^{\pm 1}) \Gamma(u\sqrt{t} z^{\pm 1}, \frac{\sqrt{pq}}{u} y^{\pm 1}) \qedhere
\end{align*}
\end{proof}

Now we can combine the three lemma's. First using the first lemma, applying the second lemma $n-2$ times and finishing with the third lemma. Let us first calculate the entire induction part from lemma 2.  Thus we get 
\begin{align*}
F_{3,n-2}(t_r;u;v) & = 
\prod_{k=3}^n \Gamma(t^{n+1-k}, \frac{pq}{t^k}, \frac{t^k}{pq}, t^{\frac12(k-1)} u t^{\frac12(k-3)} v^{\pm 1}, 
\frac{pq}{u t^{\frac12(k-3)} t^{\frac12(k+1)}} v^{\pm 1}, 
\frac{(pq)^2}{t^{k+1}})
\\ & \qquad \times \prod_{1\leq r<s\leq 7} \Gamma(t^{k-3} t_rt_s)
\prod_{r=1}^7 \Gamma( \frac{\sqrt{pq} t^{\frac12(k-3)} t_r}{u t^{\frac12(k-3)} t^{k-\frac12}})
F_{n+1,0}(t_r t^{\frac12 (n-2)}; u t^{\frac12 (n-2)};v)
%\\ &= \prod_{k=3}^n \Gamma(t^{n+1-k}, \frac{pq}{t^k}, \frac{t^k}{pq}, t^{k-2} u v^{\pm 1}, 
%\frac{pq}{u t^{k-1} } v^{\pm 1}, 
%\frac{(pq)^2}{t^{k+1}})
%\\ & \qquad \times \prod_{1\leq r<s\leq 7} \Gamma(t^{k-3} t_rt_s)
%\prod_{r=1}^7 \Gamma( \frac{\sqrt{pq}  t_r}{u  t^{k-\frac12}})
%F_{n+1,0}(t_r t^{\frac12 (n-2)}; u t^{\frac12 (n-2)};v)
\\ &= \Gamma(\frac{t^3}{pq}, \frac{(pq)^2}{t^{n+1}}, t uv^{\pm 1}, \frac{pq}{u t^{n-1}} v^{\pm 1}, t, t^2, \frac{pq}{t^n},\frac{pq}{t^{n-1}} )
\\ & \qquad \times \prod_{k=3}^n \prod_{1\leq r<s\leq 7} \Gamma(t^{k-3} t_rt_s)
\prod_{r=1}^7 \Gamma( \frac{\sqrt{pq}  t_r}{u  t^{k-\frac12}})
F_{n+1,0}(t_r t^{\frac12 (n-2)}; u t^{\frac12 (n-2)};v)
\end{align*}
Now we can easily combine the three lemma's to obtain
\begin{align*}
\int & \Delta_{II}^{(n)}(t;z)\prod_{r=1}^8 \Gamma(t_r z^{\pm 1}) \Gamma(\sqrt{\frac{pq}{t}} v^{\pm 1} z^{\pm 1}) 
\\ & \stackrel{Lem 4.2}{=} \Gamma(\sqrt{pq} t_8 t^{n-5/2} v^{\pm 1}, \frac{pq}{t^2},t^n, t^{n-1},\frac{pq}{t}, \frac{(pq)^2}{t^3}) \prod_{1\leq r<s\leq 7}  \Gamma(t_rt_s, tt_rt_s)  \\ & \qquad  \times \prod_{r=1}^7 \Gamma( \sqrt{\frac{pq}{t}} t_r v^{\pm 1}, t_rt_8 t^{n-1},  t_rt_8t^{n-2})  F_{3,n-2}(tt_1, \ldots, tt_7;\frac{\sqrt{pq}}{t_8 t^{n-3/2}};v)
\\ & \stackrel{Lem 4.3}{=} \Gamma(\sqrt{pq} t_8 t^{n-5/2} v^{\pm 1}, \frac{pq}{t^2},t^n, t^{n-1},\frac{pq}{t}, \frac{(pq)^2}{t^3}) \prod_{1\leq r<s\leq 7}  \Gamma(t_rt_s, tt_rt_s)  \\ & \qquad  \times \prod_{r=1}^7 \Gamma( \sqrt{\frac{pq}{t}} t_r v^{\pm 1}, t_rt_8 t^{n-1},  t_rt_8t^{n-2}) \\ & \qquad \times 
\Gamma(\frac{t^3}{pq}, \frac{(pq)^2}{t^{n+1}}, t \frac{\sqrt{pq}}{t_8 t^{n-3/2}} v^{\pm 1}, \sqrt{\frac{pq}{t}}t_8 v^{\pm 1}, t, t^2, \frac{pq}{t^n},\frac{pq}{t^{n-1}} )
\\ & \qquad \times \prod_{k=3}^n \prod_{1\leq r<s\leq 7} \Gamma(t^{k-1} t_rt_s)
\prod_{r=1}^7 \Gamma(t_r t_8 t^{n-k})
F_{n+1,0}(t_r t^{\frac12 n}; \frac{\sqrt{pq}}{t_8 t^{\frac12 n-1/2}};v)
\end{align*}
The calculation continues as 
\begin{align*}
 &\stackrel{Lem 4.4}{=} \Gamma( \frac{(pq)^2}{t^{n+1}},  \sqrt{\frac{pq}{t}}t_8 v^{\pm 1} )
 \prod_{r=1}^7 \Gamma( \sqrt{\frac{pq}{t}} t_r v^{\pm 1}) \prod_{k=1}^n \prod_{1\leq r<s\leq 7} \Gamma(t^{k-1} t_rt_s) \prod_{r=1}^7 \Gamma(t_r t_8 t^{n-k})
\\ & \qquad \times 
\Gamma(\frac{t^{n+1}}{pq}, t^n,  \frac{\sqrt{pqt}}{t_8} v^{\pm 1}) \prod_{1\leq r<s\leq 7} \Gamma(\frac{pq}{t^{n-1}t_rt_s}) \prod_{r=1}^7 
\Gamma(\frac{\sqrt{pqt}}{t_r} v^{\pm 1}, \frac{pq}{t_8 t_rt^{n-1}}) \\& \qquad \times 
\iint \Delta_{II}^{(n-1)}(t;z) \Delta_{I}^{(1)}(y) \Gamma(\sqrt{\frac{pq}{t}} z^{\pm 1}y^{\pm 1}) 
\prod_{r=1}^7 \Gamma(\frac{\sqrt{pq}}{t_rt^{\frac12n-1}} z^{\pm 1},   t_rt^{\frac12(n-1)} y^{\pm 1})
\\ & \qquad \times \Gamma(\sqrt{\frac{pq}{t^n}}  v^{\pm 1} y^{\pm 1}) \Gamma(\frac{\sqrt{pq}}{t_8 t^{\frac12 n-1}} z^{\pm 1}, t_8 t^{\frac12( n-1)} y^{\pm 1})
%\end{align*}
%\begin{align*}
\\ &= \Gamma( t^n)  \prod_{k=1}^{n-1} \prod_{1\leq r<s\leq 8} \Gamma(t^{k-1} t_rt_s) 
  \\& \qquad \times 
\iint \Delta_{II}^{(n-1)}(t;z) \Delta_{I}^{(1)}(y) \Gamma(\sqrt{\frac{pq}{t}} z^{\pm 1}y^{\pm 1}) 
\prod_{r=1}^8 \Gamma(\frac{\sqrt{pq}}{t_rt^{\frac12n-1}} z^{\pm 1},   t_rt^{\frac12(n-1)} y^{\pm 1})
 \Gamma(\sqrt{\frac{pq}{t^n}}  v^{\pm 1} y^{\pm 1}) 
\end{align*}
Where we simplified the expression in the final step.

The final thing to check is whether we were justified in exchanging the order of integration all those times. This is a case where we can not find parameters for which we can always use unit circle contours. For example, in the final expression of the proof of Lemma \ref{lem41}, this would imply that $|tt_r|<1$ for $1\leq r\leq 7$, $|t_8/t|<1$ and $|t^3/pq|<1$ (thus $|t|<1$). However, the balancing condition tells us that 
\[
\left(\frac{t^3}{pq} \right)^2 \frac{t_8}{t} \prod_{r=1}^7 tt_r = t^{12-(2n-3)} = t^{15-2n}.
\]
For $n\geq 8$ the left hand side is less than 1 in absolute value, while the right hand side is more than 1 in absolute value. 

Thus we need to show that in all cases the prefactor of the integral in Lemma \ref{lemmero} does not vanish. First we consider what kind of labels for edges we get (i.e. what cross terms there are). The edges coming from $\Delta_{I\!I}$-terms are always $t$, whereas the other labels are either $\sqrt{t}$ or $\sqrt{pq/t}$. In particular any product of these (associated to a closed path in the graph) will either be a strictly positive power of $(pq)$ times some power of $t$, or a strictly positive power of $t$. In particular it will never be in $p^{\mathbb{Z}_{\leq 0}}q^{\mathbb{Z}_{\leq 0}}$. 

Now, we first consider the products which involve a term of the form $t_r$ ($1\leq r\leq 8$). A quick review shows that these do not appear in cross terms, thus they only appear as labels to half-edges. In fact those labels are all of the form $t^k t_r$ or $t^k \sqrt{pq}/t_r$. Since the product of two $t_r$'s is independent of $p$ and $q$, the only option to still get a negative power of $p$ times a negative power of $q$ involving one of these terms is if we multiply a $t^k t_r$ term with a $t^k \sqrt{pq}/t_r$ term. However than we obtain a positive power of $(pq)$ times a product of labels of edges (which we classified above). In this case we also find that the resulting prefactor terms never vanish.

The terms involving $u$, in the later lemma's, are of the form $t^k u$, $t^k\sqrt{pq}/u$ and $t^k pq/u$ for some exponents $k$. The same argument as above shows that these will never lead to problems.

Finally we are left with the labels of half-edges which do not involve either $t_r$ or $u$. These always involve the parameter $v$. In the case where we only have $t^k v^{\pm 1}$ and 
$(pq)^k t^l v^{\pm 1}$ for some positive powers $k$ (though typically negative $l$), we can use the argument as before to show that we are fine. However, in the definition of $F_{k,l}$ and 
all expressions in Lemma \ref{lem43} we have terms $\sqrt{t^k/pq}v^{\pm 1}$ for $k\geq 3$. The only paths that could give problems are those that start with $\sqrt{t^k/pq} v^{+1}$ and end with $\sqrt{t^k/pq} v^{-1}$ (or vice versa). The associated product is either $t^k/(pq)$ times some positive power of $t$, which will never be in $p^{\mathbb{Z}_{\leq 0}} q^{\mathbb{Z}_{\leq 0}}$, or $t^k/pq$ times some positive power of $pq$ times an arbitrary power of $t$. This might still become 1, but fortunately we see that since, $\sqrt{pq/t}$ is the only label with negative powers of $t$ of an edge, and since $k\geq 3$, the sum of the power of $pq$ plus that of $t$ is strictly positive, and thus never 0.
\end{proof}

\section{The almost-Selberg transformation}\label{secas}
The second transformation we consider is reminiscent of the Selberg transformation \eqref{eqtrafo2}. It maps a Selberg integral with 10 parameters (thus 2 more than the integral for the actual Selberg transformation), two of which multiply to $pq/t$ to a double integral, an $n$-variate Selberg integral and a univariate integral. 

\begin{prop}\label{propas}
The following transformation holds
\begin{align*}
& \int  \Delta_{II}^{(n)}(t;z) \Gamma( \sqrt{\frac{pq}{t}} v^{\pm 1} z^{\pm 1})\prod_{r=1}^4 
\Gamma(t_r z^{\pm 1}, u_r z^{\pm 1}) 
\\ & = 
\prod_{i=0}^{n-1} \Gamma(t^i t_rt_s) \prod_{i=0}^{n-2} \Gamma(t^iu_ru_s)
\prod_{r=1}^4 \Gamma( \sqrt{\frac{pq}{t}} u_r v^{\pm 1}) 
\Gamma ( \frac{pq}{s^2\sqrt{t}}) 
\\ & \qquad \times 
\iint \Delta_{II}^{(n)}(t;z) \Delta_I^{(1)}(y) \Gamma( \sqrt{\frac{pq}{t}} y^{\pm 1} z^{\pm 1}, st^{\frac14}v^{\pm 1} y^{\pm 1})
%\\ & \qquad \times 
\prod_{r=1}^4 \Gamma( \frac{t_r}{st^{\frac14}} z^{\pm 1}) \prod_{r=1}^4 \Gamma( st^{\frac14}u_r z^{\pm 1}, \frac{\sqrt{pq}t^{\frac14}}{su_r} y^{\pm 1})
\\ &= 
\prod_{i=0}^{n-1} \Gamma(t^i t_rt_s, t^i u_ru_s) 
%\\ & \qquad \times 
\iint \Delta_{I}^{(n)}(z) \Delta_I^{(n)}(y) \Gamma( \sqrt{t} y^{\pm 1} z^{\pm 1}, \frac{\sqrt{pq} s}{t^{\frac14}}v^{\pm 1} z^{\pm 1}, 
\frac{\sqrt{pq}}{st^{\frac14}} v^{\pm 1} y^{\pm 1})
\prod_{r=1}^4 \Gamma( \frac{t_r}{st^{\frac14}} z^{\pm 1},\frac{su_r}{t^{\frac14}}y^{\pm 1})
\end{align*}
where $u_r,t_r\in \mathbb{C}^*$ satisfy the balancing conditions
\begin{equation}\label{eqdefs}
\prod_r u_r = \frac{pq}{t^{n-\frac32}s^2}, \qquad 
\prod_r t_r = \frac{pqs^2}{t^{n-\frac32}}.
\end{equation}
\end{prop}
We dubbed this formula the almost-Selberg transformation, because if we remove the terms with $v$ and the $y$-integral on the second expression, and multiply $s$ by right power of $t$ it is the Selberg transformation. Note that the final equality is just a case of applying a Dixon transformation on the $y$'s. The final expression is both uglier (more integrals) and prettier (more apparent symmetries) than the second expression, so we decided to print them both.
\begin{proof}
The proof is  prototypical of how we can use the induction enabler to prove transformations for Selberg integrals with 10 parameters, two of which multiply to $pq/t$. The general idea is of course to use induction and use the induction enabler transformation to turn the $n$-dimensional integral into an $(n-1)$-dimensional one. Proving the univariate case is an relatively easy calculation. 
For the $n$-dimensional case we first apply the induction enabler on the left hand side, use the $(n-1)$-dimensional version of the transformation on the resulting $z$-integral, do some appropriate transformation on the remaining integrals (in this case a Selberg transformation), after which we can do an inverse induction enabler transformation to end up at the desired right hand side. 

Univariately we calculate as follows (starting with the right hand side and going back)
\begin{align*}
\iint & \Delta_{II}^{(1)}(t;z) \Delta_I^{(1)}(y) \Gamma( \sqrt{\frac{pq}{t}} y^{\pm 1} z^{\pm 1}, st^{\frac14}v^{\pm 1} y^{\pm 1})
\prod_{r=1}^4 \Gamma( \frac{t_r}{st^{\frac14}} z^{\pm 1}) \prod_{r=1}^4 \Gamma( st^{\frac14}u_r z^{\pm 1}, \frac{\sqrt{pq}t^{\frac14}}{su_r} y^{\pm 1}) 
\\ &\stackrel{y:W(E_7)}{=}
\prod_{r=1}^4 \Gamma( \frac{\sqrt{pqt}}{u_r} v^{\pm 1})
\iint  \Delta_{II}^{(1)}(t;z) \Delta_I^{(1)}(y) \Gamma( s t^{1/4} y^{\pm 1} z^{\pm 1}, \sqrt{\frac{pq}{t}}v^{\pm 1} y^{\pm 1})
\prod_{r=1}^4 \Gamma( \frac{t_r}{st^{\frac14}} z^{\pm 1}) \prod_{r=1}^4 \Gamma(u_r y^{\pm 1}) 
\\ & \stackrel{z}{=}
\Gamma( s^2 \sqrt{t})
\prod_{r=1}^4 \Gamma( \frac{\sqrt{pqt}}{u_r} v^{\pm 1})
\prod_{1\leq r<s\leq 4} \frac{1}{\Gamma( t_rt_s)}
\int  \Delta_{II}^{(1)}(t;y) \Gamma( \sqrt{\frac{pq}{t}}v^{\pm 1} y^{\pm 1})
\prod_{r=1}^4 \Gamma( t_r y^{\pm 1}) \prod_{r=1}^4 \Gamma(\frac{\sqrt{pq}}{u_r} y^{\pm 1}) 
\end{align*}

Now suppose the theorem holds for $n-1$. Then we can calculate
\begin{align}
\int & \Delta_{II}^{(n)}(t;z) \Gamma( \sqrt{\frac{pq}{t}} v^{\pm 1} z^{\pm 1})\prod_{r=1}^4 
\Gamma(t_r z^{\pm 1}, u_r z^{\pm 1}) 
\nonumber \\ &\stackrel{{\color{blue} \boxtimes } :z  }{=} 
\Gamma(t^n) \prod_{i=1}^{n-1} \prod_{1\leq r<s\leq 4} \Gamma(t^{i-1} t_rt_s, t^{i-1}u_ru_s) 
\prod_{r=1}^4 \prod_{s=1}^4 \Gamma(t^{i-1} t_ru_s) 
\iint \Delta_{II}^{(n-1)}(t;z) \Delta_I^{(1)}(y) 
\label{eqbbb}  \\ & \qquad \times 
\Gamma(\sqrt{\frac{pq}{t}} y^{\pm 1} z^{\pm 1}, \sqrt{\frac{pq}{t^n}} v^{\pm 1} y^{\pm 1}) \prod_{r=1}^4 \Gamma( \frac{\sqrt{pq}}{t_r t^{\frac12 n-1}} z^{\pm 1}, \frac{\sqrt{pq}}{u_r t^{\frac12 n-1}} z^{\pm 1}, t_r t^{\frac12(n-1)} y^{\pm 1},u_rt^{\frac12(n-1)} y^{\pm 1})
\nonumber \\ & \stackrel{IH:z}{=} 
\Gamma(t^n) \prod_{1\leq r<s\leq 4} \Gamma( t_rt_s)  \prod_{i=1}^{n-1}
\prod_{r=1}^4 \prod_{s=1}^4 \Gamma(t^{i-1} t_ru_s) 
\Gamma( \frac{pq}{ s^2 \sqrt{t} })
\iiint \Delta_{II}^{(n-1)}(t;z) \Delta_I^{(1)}(y) \Delta_I^{(1)}(x) \Gamma(\sqrt{\frac{pq}{t}} x^{\pm 1} z^{\pm 1})
\nonumber \\ & \qquad \times 
\Gamma( st^{\frac14} y^{\pm 1} x^{\pm 1}, \sqrt{\frac{pq}{t^n}} v^{\pm 1} y^{\pm 1}) \prod_{r=1}^4 \Gamma( \frac{\sqrt{pq} s}{t_r t^{\frac12 n-\frac54}} z^{\pm 1}, \frac{\sqrt{pq}}{u_r t^{\frac12 n-\frac34}s } z^{\pm 1}, \frac{t_r t^{\frac12 n-\frac34} }{s  } x^{\pm 1}, u_rt^{\frac12(n-1)} y^{\pm 1})
\label{eqIH} \\ & \stackrel{W(E_7):y}{=}
\Gamma(t^n) \prod_{1\leq r<s\leq 4} \Gamma( t_rt_s)  \prod_{i=1}^{n-1}
\prod_{r=1}^4 \prod_{s=1}^4 \Gamma(t^{i-1} t_ru_s) 
\Gamma( \frac{pq}{ s^2 \sqrt{t} })
\prod_{r=1}^4 \Gamma( \sqrt{\frac{pq}{t}} u_r v^{\pm 1})
\nonumber \\ & \qquad \times 
\iiint \Delta_{II}^{(n-1)}(t;z) \Delta_I^{(1)}(y) \Delta_I^{(1)}(x) 
\Gamma(\sqrt{\frac{pq}{t}} x^{\pm 1} z^{\pm 1},\sqrt{\frac{pq}{t^n}}  y^{\pm 1} x^{\pm 1}, st^{\frac14}  v^{\pm 1} y^{\pm 1}) 
\nonumber \\ & \qquad \times \prod_{r=1}^4 \Gamma( \frac{\sqrt{pq} s}{t_r t^{\frac12 n-\frac54}} z^{\pm 1}, \frac{\sqrt{pq}}{u_r t^{\frac12 n-\frac34}s } z^{\pm 1}, \frac{t_r t^{\frac12 n-\frac34} }{s  } x^{\pm 1}, s u_rt^{\frac12n-\frac14} x^{\pm 1},\frac{\sqrt{pq} t^{\frac14}}{s u_r} y^{\pm 1})
\label{eqX} \\ &\stackrel{ {\color{blue} \boxtimes } :z  }{=} 
\prod_{i=0}^{n-1}  \prod_{1\leq r<s\leq 4} \Gamma( t^{i} t_rt_s)
\prod_{i=0}^{n-2} \prod_{1\leq r<s\leq 4} \Gamma(t^{i} u_ru_s  )
\Gamma( \frac{pq}{ s^2 \sqrt{t} })
\prod_{r=1}^4 \Gamma( \sqrt{\frac{pq}{t}} u_r v^{\pm 1})
\nonumber \\ & \qquad \times 
\iint \Delta_{II}^{(n)}(t;z) \Delta_I^{(1)}(y) 
\Gamma(\sqrt{\frac{pq}{t}} y^{\pm 1} z^{\pm 1},st^{\frac14}  v^{\pm 1} y^{\pm 1}) 
%\\ & \qquad \times
 \prod_{r=1}^4 \Gamma(\frac{t_r  }{s t^{\frac14} } z^{\pm 1}, s u_rt^{\frac14} z^{\pm 1},\frac{\sqrt{pq} t^{\frac14}}{s u_r} y^{\pm 1}) 
 \nonumber 
\end{align}
Here we use the induction hypothesis (in the $IH$ step) for $z$ with $\frac{\sqrt{pq}}{u_r t^{\frac12 n-1}}$ in the role of the $t_r$'s. 

Applying Fubini is not a problem in this proof, as we can find parameters for which we can always use unit circle contours. Indeed if $|pq|\ll |t^n|$, $|t|<1$ and $s=1$, we can choose $u_r = (pq)^{\frac14} t^{-\frac14n+\frac38}$, so both $u_r$, and $\sqrt{pq}/u_r$ are very small, and thus less than 1, even if multiplied with some power of $t$. The same holds for the $t_r$. The remaining parameters are $\sqrt{pq/t^n}$, $\sqrt{pq/t}$ (which are also very small), $t$ and $st^{\frac14}$, the latter are fine as $|t|<1$ by assumption. Thus we can use this proof in an non-empty open set of parameters, and by analytical extension it holds in general.
\end{proof}

We will actually mostly use the following lemma, which follows directly from the proof of the almost Selberg transformation.
\begin{lemma}\label{lem52}
Assuming the  conditions of Proposition \ref{propas} we have the following equation
\begin{align*}
& \int  \Delta_{I\!I}^{(n)}(t;z) \Gamma(\sqrt{\frac{pq}{t}} v^{\pm 1} z^{\pm 1})
\prod_{r=1}^4 \Gamma(t_r z^{\pm 1}, u_rz^{\pm 1})
\\ &=\Gamma(t^n,t^{n-1}, \frac{pq}{s^2\sqrt{t}}) 
\prod_{1\leq r<s\leq 4} \prod_{k=0}^{n-2} \Gamma(t^k t_rt_s) \prod_{k=0}^{n-3} 
\Gamma(t^k u_ru_s)
\prod_{r,s=1}^4 \Gamma(t_r u_s)
\\ & \qquad \times 
\iiiint \Delta_{I\!I}^{(n-2)}(t;z) \Delta_I^{(1)}(y) \Delta_I^{(1)}(x) \Delta_I^{(1)}(w) \Gamma(\sqrt{\frac{pq}{t}} w^{\pm 1}z^{\pm 1},
\sqrt{\frac{pq}{t^{n-1}}} x^{\pm 1} w^{\pm 1},
 st^{\frac14}  x^{\pm 1}y^{\pm 1},\sqrt{\frac{pq}{t^n}} v^{\pm 1} y^{\pm 1})
\\ & \qquad \times 
\prod_{r=1}^4 \Gamma( \frac{t_r t^{\frac14}}{ s} z^{\pm 1}, 
s u_r t^{ \frac34}  z^{\pm 1},
 \frac{\sqrt{pq} s t^{\frac14}}{t_r}  w^{\pm 1}, 
\frac{\sqrt{pq}}{s u_r t^{ \frac14} }  w^{\pm 1},
\frac{t_r}{s} t^{\frac12 n-\frac34} x^{\pm 1},  u_r  t^{\frac12n-\frac12}  y^{\pm 1}) 
\end{align*}
\end{lemma}
\begin{proof}
We take the equality obtained in the proof of the almost Selberg transformation, which relates the left hand side to the result of applying the induction hypothesis (i.e. to the expression \eqref{eqIH}) and subsequently apply an induction enabler transformation (in the $z$-variable) to obtain the desired right hand side.
\end{proof}

\section{The first quadratic transformation}
In this section we prove the first of two quadratic transformations. This is Conjecture Q3 in \cite{RainsLittlewood}, in the case $\lambda=0$. The univariate case was already proven by Rains in that paper. This quadratic transformation is an equation between $(p,q)$-elliptic hypergeometric functions and $(p,q^2)$-elliptic hypergeometric functions. Recall the notation that 
$\tilde \Gamma(x) = \Gamma_{p,q^2}(x)$, while $\Gamma(x)=\Gamma_{p,q}(x)$, with similar notations for the kernels $\Delta$. The gamma functions for these pairs of parameters are related by 
\[
\Gamma(x)=  \tilde \Gamma(x,qx), \qquad \qquad 
\tilde \Gamma(x^2) = \Gamma_{p^2,q^2}(x^2,px^2) = \Gamma(\pm x, \pm p^{1/2} x),
\]
which follows immediately from the product expression for the elliptic gamma function. We will use these equations constantly without notification.

The result is the following 
\begin{thm}\label{thmq5}
Under the balancing condition $t^{2n-1}t_1t_2t_3t_4=pq^2$ the equation 
\begin{align*}
\int  \Delta_{I\!I}^{(n)}(t;z)  &
\prod_{r=1}^4 \Gamma(t_r z^{\pm 1}) \prod_{r=1}^2 \Gamma(\sqrt{\frac{pq}{t}} v^{\pm 1} z^{\pm 1})
\tilde \Gamma(t z^{\pm 2})
\\ & = 
\prod_{i=0}^{n-1} \prod_{1\leq r<s\leq 4}\Gamma(t^{2i} t_r t_s)
\int \tilde \Delta_{I\!I}^{(n)}(t^2;z)
\prod_{r=1}^4 \Gamma(\sqrt{\frac{t}{q}}  t_r z^{\pm 1}) \prod_{r=1}^2 \tilde \Gamma( \sqrt{\frac{pq^2}{t^2}} v^{\pm 1} z^{\pm 1})
\end{align*}
holds. 
\end{thm}
Note that both sides can be expressed as Selberg integrals with 10 parameters (specialized in the right way). 
\begin{proof}
For completeness purposes we recall the proof of the univariate case from \cite[Proposition 5.6]{RainsLittlewood}.
Indeed in that case we consider the double integral
\[
\iint \Delta_I^{(1)}(z) \tilde \Delta_I^{(1)}(y) \Gamma(\sqrt{\frac{t}{q}} y^{\pm 1} z^{\pm 1}) 
\tilde \Gamma( \sqrt{\frac{pq^2}{t^2}} v^{\pm 1} y^{\pm 1}) \prod_{r=1}^4 \Gamma(t_r z^{\pm 1})
\]
and evaluate it either to $y$ or to $z$, which leads to the desired equation.

The more general case can now be obtained using induction and the following calculation
\begin{align*}
\int & \tilde \Delta_{II}^{(n)}(t^2;z) \prod_{r=1}^4 \Gamma( \sqrt{\frac{t}{q}} t_r z^{\pm 1}) \tilde \Gamma( \sqrt{\frac{pq^2}{t^2}} v^{\pm 1} z^{\pm 1}) 
\\ &\stackrel{{\color{blue} \boxtimes}:z}{=}
\tilde \Gamma(t^{2n})  \prod_{k=1}^{n-1} \prod_{1\leq r<s\leq 4} \Gamma( \frac{t^{2k-1} }{q} t_rt_s, t^{2k-1} t_rt_s) \prod_{r=1}^4 \tilde \Gamma(t^{2k-1}t_r^2)
\\ & \qquad \times 
\iint \tilde \Delta_{II}^{(n-1)}(t^2;z)\tilde \Delta_I^{(1)}(y) \tilde \Gamma( \sqrt{\frac{pq^2}{t^2}} y^{\pm 1} z^{\pm 1}, \sqrt{\frac{pq^2}{t^{2n}}} v^{\pm 1} y^{\pm 1}) 
\prod_{r=1}^4 \Gamma( \frac{t_r t^{n-\frac12}}{\sqrt{q}} y^{\pm 1}, \frac{\sqrt{pq}}{t_r t^{n-\frac32}} z^{\pm 1})
\\ & \stackrel{IH:z}{=} 
\tilde \Gamma(t^{2n})  \prod_{k=1}^{n-1} \prod_{1\leq r<s\leq 4} \Gamma( \frac{t^{2k-1} }{q} t_rt_s) \prod_{r=1}^4 \tilde \Gamma(t^{2k-1}t_r^2)
\\ & \qquad \times 
\iint \Delta_{II}^{(n-1)}(t;z)\tilde \Delta_I^{(1)}(y) \Gamma( \sqrt{\frac{pq}{t}} y^{\pm 1} z^{\pm 1}) \tilde \Gamma(\sqrt{\frac{pq^2}{t^{2n}}} v^{\pm 1} y^{\pm 1}) 
\prod_{r=1}^4 \Gamma( \frac{t_r t^{n-\frac12}}{\sqrt{q}} y^{\pm 1}, \frac{\sqrt{p}q}{t_r t^{n-1}} z^{\pm 1})
\tilde \Gamma(t z^{\pm 2})
\\ & \stackrel{{\color{blue} \boxtimes}:z}{=} 
\tilde \Gamma(t^{2n})  \Gamma(t^{n-1}) \prod_{k=1}^{n-1} \prod_{1\leq r<s\leq 4} \Gamma( \frac{t^{2k-1} }{q} t_rt_s) \prod_{r=1}^4 \tilde \Gamma(t^{2k-1}t_r^2)
\\ & \qquad \times \prod_{k=1}^{n-2} \prod_{1\leq r<s\leq 4}  \Gamma( \frac{pq^2}{t_rt_s t^{2n-k-1}}) 
\prod_{r=1}^4 \tilde \Gamma( \frac{pq^2}{t_r^2 t^{2n-2k-1}})
\frac{\tilde \Gamma( t^{2k}, p t^{2k})}{\Gamma(t^{2k},pt^{2k})}
\\ & \qquad \times 
\iiint  \Delta_{II}^{(n-2)}(t;z)\tilde \Delta_I^{(1)}(y)\Delta_I^{(1)}(x) 
\Gamma( \sqrt{\frac{pq}{t}} x^{\pm 1} z^{\pm 1}, \sqrt{\frac{pq}{t^{n-1}}} y^{\pm 1}x^{\pm 1}) \tilde \Gamma(\sqrt{\frac{pq^2}{t^{2n}}} v^{\pm 1} y^{\pm 1}) 
\\ & \qquad \times \prod_{r=1}^4 \Gamma( \frac{t_r t^{n-\frac12}}{\sqrt{q}} y^{\pm 1}, 
\frac{\sqrt{p}q}{t_r t^{\frac12n}} x^{\pm 1}, \frac{t_r t^{\frac12 n+\frac12}  }{\sqrt{q} } z^{\pm 1})
\tilde \Gamma(t^{n-1} x^{\pm 2}, \frac{q}{t^{n-2}} z^{\pm 2})
\\ & \stackrel{w}{=} \tilde \Gamma(t^{2n})  \frac{\Gamma(t^{n-1})}{\Gamma(  \frac{t^{2n-1}}{q})} \prod_{k=1}^{n-1} \prod_{1\leq r<s\leq 4} \Gamma( \frac{t^{2k-1} }{q} t_rt_s) \prod_{r=1}^4 \tilde \Gamma(t^{2k-1}t_r^2)
\\ & \qquad \times \prod_{k=1}^{n-2} \prod_{1\leq r<s\leq 4}  \Gamma( \frac{pq^2}{t_rt_s t^{2n-k-1}}) 
\prod_{r=1}^4 \tilde \Gamma( \frac{pq^2}{t_r^2 t^{2n-2k-1}})
\frac{\tilde \Gamma( t^{2k}, p t^{2k})}{\Gamma(t^{k},pt^{k})}
 \frac{1}{\prod_{1\leq r<s\leq 4} \Gamma(t_rt_s)}
\\ & \qquad \times 
\iiiint  \Delta_{II}^{(n-2)}(t;z)\tilde \Delta_I^{(1)}(y)\Delta_I^{(1)}(x) \Delta_I^{(1)}(w)
\Gamma( \sqrt{\frac{pq}{t}} x^{\pm 1} z^{\pm 1}, \sqrt{\frac{pq}{t^{n-1}}} y^{\pm 1}x^{\pm 1},\frac{t^{n-\frac12}}{\sqrt{q}} w^{\pm 1} y^{\pm 1}) 
\\ & \qquad \times \tilde \Gamma(\sqrt{\frac{pq^2}{t^{2n}}} v^{\pm 1} y^{\pm 1})  \prod_{r=1}^4 \Gamma( t_r w^{\pm 1},  \frac{\sqrt{p}q}{t_r t^{\frac12n}} x^{\pm 1}, \frac{t_r t^{\frac12 n+\frac12}  }{\sqrt{q} } z^{\pm 1})
\tilde \Gamma(t^{n-1} x^{\pm 2}, \frac{q}{t^{n-2}} z^{\pm 2}).
%\end{align*}
%Continuing the calculation we get 
%\begin{align*}
\\ &\stackrel{IH:y}{=}
 \Gamma(t^{n})   \prod_{k=1}^{n-1} \prod_{1\leq r<s\leq 4}\Gamma( \frac{t^{2k-1} }{q} t_rt_s) \prod_{r=1}^4 \tilde \Gamma(t^{2k-1}t_r^2)
\\& \qquad \times 
\prod_{k=1}^{n-2} \prod_{1\leq r<s\leq 4}  \Gamma( \frac{pq^2}{t_rt_s t^{2n-k-1}}) 
\prod_{r=1}^4 \tilde \Gamma( \frac{pq^2}{t_r^2 t^{2n-2k-1}})
\frac{\tilde \Gamma( t^{2k}, p t^{2k})}{\Gamma(t^k,pt^k)}
\frac{1}{\prod_{1\leq r<s\leq 4} \Gamma(t_rt_s)}
\\ & \qquad \times \iiiint \Delta_{II}^{(n-2)}(t;z) \Delta_I^{(1)}(y)\Delta_I^{(1)}(x) \Delta_I^{(1)}(w)
\Gamma( \sqrt{\frac{pq}{t}} x^{\pm 1} z^{\pm 1}, \sqrt{\frac{pq^2}{t^{2n-1}}} y^{\pm 1}x^{\pm 1},t^{\frac12n-\frac12} w^{\pm 1} y^{\pm 1}) 
\\ & \qquad \times
\Gamma(\sqrt{p} t^{\frac12 n}  x^{\pm 1} w^{\pm 1},\sqrt{\frac{pq}{t^{n}}} v^{\pm 1} y^{\pm 1}) 
 \prod_{r=1}^4 \Gamma( t_r w^{\pm 1},  \frac{\sqrt{p}q}{t_r t^{\frac12n}} x^{\pm 1}, \frac{t_r t^{\frac12 n+\frac12}  }{\sqrt{q} } z^{\pm 1})
\tilde \Gamma(t^{n-1} x^{\pm 2}, \frac{q}{t^{n-2}} z^{\pm 2}, t^n y^{\pm 2})
\end{align*}
Continuing as 
\begin{align*}
\\ & \stackrel{W(E_7):w}{=}
 \Gamma(t^{n})   \Gamma(t^{n-1}, p t^n) \prod_{k=1}^{n-1} \prod_{1\leq r<s\leq 4}\Gamma( \frac{t^{2k-1} }{q} t_rt_s) \prod_{r=1}^4 \tilde \Gamma(t^{2k-1}t_r^2)
\\ & \quad \times \prod_{k=1}^{n-2} \prod_{1\leq r<s\leq 4}  \Gamma( \frac{pq^2}{t_rt_s t^{2n-k-1}}) 
\prod_{r=1}^4 \tilde \Gamma( \frac{pq^2}{t_r^2 t^{2n-2k-1}})
\frac{\tilde \Gamma( t^{2k}, p t^{2k})}{\Gamma(t^k,pt^k)}
\\ & \quad \times 
\iiiint \Delta_{II}^{(n-2)}(t;z) \Delta_I^{(1)}(y)\Delta_I^{(1)}(x) \Delta_I^{(1)}(w)
\Gamma( \sqrt{\frac{pq}{t}} x^{\pm 1} z^{\pm 1}, 
\sqrt{ \frac{q}{ t^{n}}} w^{\pm 1} y^{\pm 1},\sqrt{\frac{pq}{t^{n}}} v^{\pm 1} y^{\pm 1},
\sqrt{\frac{p  q }{t^{n-1}}} x^{\pm 1} w^{\pm 1}) 
\\ & \quad \times \prod_{r=1}^4 \Gamma(\sqrt{\frac{t^{2n-1}}{q}} t_r w^{\pm 1},  \frac{\sqrt{p}q}{t_r t^{\frac12n}} x^{\pm 1}, \frac{t_r t^{\frac12 n+\frac12}  }{\sqrt{q} } z^{\pm 1})
\tilde \Gamma(t^{n-1} x^{\pm 2}, \frac{q}{t^{n-2}} z^{\pm 2}, t^n y^{\pm 2})
\end{align*}
Applying Lemma \ref{lem52} to this final expression gives us the desired left hand side. The $u$-parameters in Lemma \ref{lem52} should be specialized as 
$\pm \sqrt{t}$ and $\pm \sqrt{pt}$, while the $t_r$'s are just the $t_r$'s.

As for applications of Fubini, it is quite easy to find a non-empty open set of parameters for which the contours can always be chosen to be unit circles. Indeed the conditions that we can take a unit circle contour are saying that the parameters inside the integrals must be strictly less than 1, so it surely defines an open set, and taking $p$ very small, $t_r= p^{\frac14}q^{\frac12} t^{-\frac12 n+\frac14}$, and $t=q^{\frac{3}{4n}}$ we see that these conditions are always satisfied.
\end{proof}

\section{The second quadratic transformation}
The second quadratic transformation is a transformation relating a ($p,q$)-elliptic hypergeometric integral with an $(\sqrt{p},\sqrt{q})$-elliptic hypergeometric integral. The theorem given is Conjecture Q7 from \cite{RainsLittlewood} in the case $\lambda=0$. As a  corollary we also obtain the $\lambda=0$ version of Conjecture Q1 in loc. cit. 

Recall our notation of $\hat \Gamma$ for $\Gamma_{\sqrt{p},\sqrt{q}}$. The relevant quadratic transformations for gamma functions are 
\[
\Gamma(x) = \hat \Gamma( \pm \sqrt{x}), \qquad 
\hat \Gamma(x) = \Gamma(x,\sqrt{q} x,\sqrt{p} x, \sqrt{pq} x).
\]

\begin{thm}\label{thm71}
Under the balancing conditions $t^{2n-1}t_1t_2t_3t_4=pq $  we have
\begin{multline}\label{eqthm71}
\int \Delta_{I\!I}^{(n)}(t;z) 
\prod_{r=1}^4 \Gamma(t_r z^{\pm 1})  \Gamma(- \sqrt{\frac{pq}{t}}  v^{\pm 2} z^{\pm 1}) \hat \Gamma(\sqrt{t} z^{\pm 1}) 
 \\ =
\prod_{i=0}^{n-1} \prod_{1\leq r<s\leq 4} \Gamma(t^{2i} t_rt_s)
\int \hat \Delta_{I\!I}^{(n)}(t;z)
 \prod_{r=1}^4 \Gamma(-t_r \sqrt{t} z^{\pm 2})  \hat \Gamma(\sqrt{\frac{\sqrt{pq}}{t}} v^{\pm 1} z^{\pm 1}) 
\end{multline}
\end{thm}
Again this is an equation between two Selberg integrals with 10 parameters.
\begin{proof}
The proof is very similar to the proof of Theorem \ref{thmq5}. As in that proof the univariate case was already proven by Rains as \cite[Proposition 5.12]{RainsLittlewood}. Once again we repeat the proof for completeness. Indeed it follows from evaluating the double integral
\[
\iint \Delta_{I}^{(1)}(z) \hat \Delta_I^{(1)}(y) \Gamma(-\sqrt{t} y^{\pm 2} z^{\pm 1})
\prod_{r=1}^4 \Gamma(t_r z^{\pm 1}) \hat \Gamma( \sqrt{\frac{\sqrt{pq}}{t}} v^{\pm 1} y^{\pm 1})
\]
to $y$ and $z$.

For general $n$ we now use induction and the following calculation. 
We start with the right hand side and get 
\begin{align}
\prod_{i=0}^{n-1} & \prod_{1\leq r<s\leq 4} \Gamma(t^{2i} t_rt_s)
\int \hat \Delta_{I\!I}^{(n)}(t;z)
 \prod_{r=1}^4 \Gamma(-t_r \sqrt{t} z^{\pm 2}) \prod_{r=1}^2 \hat \Gamma(\sqrt{\frac{\sqrt{pq}}{t}} v^{\pm 1} z^{\pm 1}) 
\nonumber\\ & \stackrel{{ \color{blue} \boxtimes}:z}{=} 
\hat \Gamma(t^n)  \prod_{k=1}^{n-1} \prod_{1\leq r<s\leq 4} \Gamma(t^{2k-1} t_rt_s) \prod_{r=1}^4  \hat \Gamma(t^{k-\frac12}t_r)
 \prod_{1\leq r<s\leq 4} \Gamma(t_rt_s)
\nonumber \\ & \qquad  \times \iint \hat \Delta_{I\!I}^{(n-1)}(t;z) \hat \Delta_I^{(1)}(y)
 \hat \Gamma(\left(\frac{pq}{t^2}\right)^{\frac14} y^{\pm 1} z^{\pm 1}) 
\prod_{r=1}^4 \Gamma(-\frac{\sqrt{pq}}{t_r t^{n-\frac32}} z^{\pm 2},-t_r t^{n-\frac12} y^{\pm 2})  
\hat \Gamma(\left(\frac{pq}{t^{2n}}\right)^{\frac14} v^{\pm 1} y^{\pm 1}) 
\nonumber \\ & \stackrel{IH:z}{=}
\hat \Gamma(t^n)  \prod_{k=1}^{n-1} \prod_{r=1}^4  \hat \Gamma(t^{k-\frac12}t_r)
\prod_{1\leq r<s\leq 4} \Gamma(t_rt_s)
\iint \Delta_{I\!I}^{(n-1)}(t;z) \hat \Delta_I^{(1)}(y)
\nonumber \\ & \qquad  \times 
  \Gamma(-\sqrt{\frac{pq}{t}}  y^{\pm 2} z^{\pm 1}) 
\prod_{r=1}^4 \Gamma(\frac{\sqrt{pq}}{t_r t^{n-1}} z^{\pm 1},-t_r t^{n-\frac12} y^{\pm 2})  
\hat \Gamma(\left(\frac{pq}{t^{2n}}\right)^{\frac14} v^{\pm 1} y^{\pm 1}) \hat \Gamma(\sqrt{t} z^{\pm 1})
\label{eqIII} \\ & \stackrel{z\to -z}{=} 
\hat \Gamma(t^n)  \prod_{k=1}^{n-1} \prod_{r=1}^4  \hat \Gamma(t^{k-\frac12}t_r)
\prod_{1\leq r<s\leq 4} \Gamma(t_rt_s)
\iint \Delta_{I\!I}^{(n-1)}(t;z) \hat \Delta_I^{(1)}(y)
\nonumber \\ & \qquad \times 
  \Gamma(\sqrt{\frac{pq}{t}}  y^{\pm 2} z^{\pm 1}) 
\prod_{r=1}^4 \Gamma(-\frac{\sqrt{pq}}{t_r t^{n-1}} z^{\pm 1},-t_r t^{n-\frac12} y^{\pm 2})  
\hat \Gamma(\left(\frac{pq}{t^{2n}}\right)^{\frac14} v^{\pm 1} y^{\pm 1}) \hat \Gamma(-\sqrt{t} z^{\pm 1})
\nonumber
%\end{align}
%The calculation continues as  
%\begin{align}
\\  & \stackrel{{\color{blue} \boxtimes}:z}{=}
\hat \Gamma(t^n)  \Gamma(t^{n-1}) 
\prod_{r=1}^4 \hat \Gamma( t^{\frac12} t_r)
\prod_{1\leq r<s\leq 4} \Gamma(t_rt_s) 
 \prod_{k=1}^{n-2} \prod_{1\leq r<s\leq 4} \Gamma(t^k t_rt_s) \frac{\hat \Gamma(t^k, \sqrt{pq} t^k)}{\Gamma(t^k,pqt^k)}
\nonumber \\ & \qquad \times 
\iiint  \Delta_{I\!I}^{(n-2)}(t;z) \hat \Delta_I^{(1)}(y) \Delta_{I}^{(1)}(x)
 \Gamma(\sqrt{\frac{pq}{t}}  x^{\pm 1} z^{\pm 1}) 
\prod_{r=1}^4 \Gamma(-t_r t^{\frac12 n+\frac12}  z^{\pm 1},-\frac{\sqrt{pq}}{t_r t^{\frac12 n}} x^{\pm 1} ,-t_r t^{n-\frac12} y^{\pm 2})  
\nonumber \\ & \qquad \times 
\hat \Gamma(\left(\frac{pq}{t^{2n}}\right)^{\frac14} v^{\pm 1} y^{\pm 1}) 
\Gamma(\sqrt{\frac{pq}{t^{n-1}}} x^{\pm 1} y^{\pm 2})
\hat \Gamma( - t^{1-\frac12n}z^{\pm 1},- t^{\frac12n-\frac12} x^{\pm 1})
\nonumber \\ & \stackrel{w}{=}
\hat \Gamma(t^n)  \frac{\Gamma(t^{n-1})}{\Gamma(t^{2n-1})} 
\prod_{r=1}^4 \hat \Gamma(t^{\frac12} t_r)
 \prod_{k=1}^{n-2} \prod_{1\leq r<s\leq 4} \Gamma(t^kt_rt_s ) \frac{\hat \Gamma(t^k, \sqrt{pq}t^k)}{\Gamma(t^k,pqt^k)}
\nonumber \\ & \qquad \times 
\iiiint  \Delta_{I\!I}^{(n-2)}(t;z) \hat \Delta_I^{(1)}(y) \Delta_{I}^{(1)}(x) \Delta_{I}^{(1)}(w) 
 \Gamma(\sqrt{\frac{pq}{t}}  x^{\pm 1} z^{\pm 1},t^{n-\frac12} w^{\pm 1} y^{\pm 2},\sqrt{\frac{pq}{t^{n-1}}} x^{\pm 1} y^{\pm 2})
\nonumber \\ & \qquad \times 
 \hat \Gamma(\left(\frac{pq}{t^{2n}}\right)^{\frac14} v^{\pm 1} y^{\pm 1}) 
\prod_{r=1}^4 \Gamma(-t_r t^{\frac12 n+\frac12}  z^{\pm 1},-\frac{\sqrt{pq}}{t_r t^{\frac12 n}} x^{\pm 1}
,-t_r w^{\pm 1})  
\hat \Gamma( - t^{1-\frac12n}z^{\pm 1}, -t^{\frac12n-\frac12} x^{\pm 1})
\nonumber \\& \stackrel{IH:y}{=}
\Gamma(t^n)  \prod_{r=1}^4 \hat \Gamma(t^{\frac12} t_r) \prod_{k=1}^{n-2} \prod_{1\leq r<s\leq 4} \Gamma(t^kt_rt_s)  
\frac{\hat \Gamma(t^k,\sqrt{pq}t^k)}{\Gamma(t^k,pqt^k)}
\nonumber \\ & \qquad \times 
\iiiint  \Delta_{I\!I}^{(n-2)}(t;z) \Delta_I^{(1)}(y) \Delta_{I}^{(1)}(x) \Delta_{I}^{(1)}(w) 
\nonumber \\ & \qquad \times 
 \Gamma(\sqrt{\frac{pq}{t}}  x^{\pm 1} z^{\pm 1},  \sqrt{pq t^n} x^{\pm 1} w^{\pm 1},-t^{\frac12 n-\frac12} w^{\pm 1} y^{\pm 1},-\sqrt{\frac{pq}{t^{2n-1}}}x^{\pm 1} y^{\pm 1})
\label{eqVII} \\ & \qquad \times \prod_{r=1}^4 \Gamma(-t_r t^{\frac12 n+\frac12}  z^{\pm 1},-\frac{\sqrt{pq}}{t_r t^{\frac12 n}} x^{\pm 1}
,-t_r w^{\pm 1})  
\Gamma(-\sqrt{\frac{pq}{t^n}} v^{\pm 2} y^{\pm 1}) 
\hat \Gamma(-  t^{1-\frac12n}z^{\pm 1}, -t^{\frac12n-\frac12} x^{\pm 1}, t^{\frac12n} y^{\pm 1})
\nonumber 
\end{align}
Continuing as 
\begin{align}
 & \stackrel{W(E_7):w}{=}
\Gamma(t^n, pq t^n, t^{n-1})  
\prod_{r=1}^4 \hat \Gamma(t^{\frac12} t_r) \prod_{k=0}^{n-2}\prod_{1\leq r<s\leq 4} \Gamma(t^kt_rt_s)
\prod_{k=1}^{n-2} \frac{\hat \Gamma(t^k,\sqrt{pq}t^k)}{\Gamma(t^k,pqt^k)}
\nonumber \\ & \qquad \times 
\iiiint  \Delta_{I\!I}^{(n-2)}(t;z) \Delta_I^{(1)}(y) \Delta_{I}^{(1)}(x) \Delta_{I}^{(1)}(w) 
 \Gamma(\sqrt{\frac{pq}{t}}  x^{\pm 1} z^{\pm 1},  \sqrt{\frac{pq}{t^{n-1}}} x^{\pm 1} w^{\pm 1},- t^{-\frac12 n} w^{\pm 1} y^{\pm 1})
\nonumber \\ & \qquad \times \prod_{r=1}^4 \Gamma(-t_r t^{\frac12 n+\frac12}  z^{\pm 1},-\frac{\sqrt{pq}}{t_r t^{\frac12 n}} x^{\pm 1}
,-t_rt^{n-\frac12} w^{\pm 1})  
\Gamma(-\sqrt{\frac{pq}{t^n}} v^{\pm 2} y^{\pm 1}) 
\hat \Gamma( - t^{1-\frac12n}z^{\pm 1},- t^{\frac12n-\frac12} x^{\pm 1}, t^{\frac12n} y^{\pm 1})
\nonumber 
\end{align}
After replacing $z\to -z$, $x\to -x$, $w\to -w$ we can apply Lemma \ref{lem52} to equate the final expression to the desired left hand side. The $v$ in Lemma \ref{lem52} should be replaced by $-v^2$, the $u$-parameters should specialized to $\sqrt{t}$, $\sqrt{qt}$, $\sqrt{pt}$ and $\sqrt{pqt}$, while the $t_r$ remain just $t_r$.

There is no set of parameters for which we can choose unit circle contours everywhere (for example in this last expression we would need both $|t^{-\frac12n}|<1$ and $|t^{\frac12 n}|<1$), thus we need to consider the prefactor in Lemma \ref{lemmero}. The labels of the edges and half-edges only have positive powers of $pq$, so whenever we use an edge with a strictly positive power of $pq$ the product will never be an element of $p^{\mathbb{Z}_{\leq 0}}q^{\mathbb{Z}_{\leq 0}}$ (and for the rest we only need to worry that the product over a path equals 1). If we ignore the labels with a strictly positive power of $pq$, we see that $t_r$ also only occurs in positive powers, so we can ignore labels with a $t_r$ as well. What is left are labels which are powers of $t$ (possibly with a sign). Up till the second application of the induction enabler transformation, all powers of $t$ are positive and we have no problems. After that however we have a term $\Delta_{I\!I}^{(n-2)}(t;z) \hat \Gamma(-t^{1-\frac12n}z^{\pm 1})$, which is problematic, as $t^{1-\frac12n}\cdot t^{n-2} \cdot t^{1-\frac12n}=1$ corresponds to a path of length $n$ starting and ending at this half-edge and moving between the vertices associated to the different $z_i$'s in the graph. In the final equation we have another problem with the terms coming from 
$\Gamma(t^{-\frac12 n} w^{\pm 1} y^{\pm 1}, t^{\frac12n} y^{\pm 1})$. In order to resolve these issues we have to use a trick. 

The trick is do parts of the calculation for more general parameters. The less explicit the coefficients of the elliptic gamma functions, the less chance there is that we would seem to be at an apparent pole of the integral. We can thus prove lemma's showing the equality of two integrals for more general parameters, and as long as the specialization of these two integrals is valid, we can equate them for these special variables. In particular the specializations of the intermediate steps do not have to be valid, that is, we might be specializing at apparent poles. It turns out to be a non-trivial exercise to find the right intermediate lemma's and in particular we have to make an, otherwise unnecessary, detour. First we prove a doubly-generalized version of the equality between \eqref{eqVII} and \eqref{eqbbb} (the latter arising in the proof of Lemma \ref{lem52}). Note that we can specialize both ends at $b_r\to \sqrt{t}, \sqrt{pt}, \sqrt{qt}, \sqrt{pqt}$, while this is not valid for the intermediate steps. We also want to specialize $u_r\to \sqrt{t}, \sqrt{pt},\sqrt{qt},\sqrt{pqt}$, but that is impossible for these two integrals.
\begin{lemma}
Under the balancing conditions $t^{n-\frac32}\prod_{r=1}^4 t_r=pqs^2$, $t^{n-\frac32} s^2 \prod_{r=1}^4 u_r=pq$ and $\prod_{r=1}^4 u_r=\prod_{r=1}^4 b_r$ we have
\begin{align*}
\iiiint & \Delta_{I\!I}^{(n-2)}(t;z) \Delta_I^{(1)}(y)\Delta_I^{(1)}(x)\Delta_I^{(1)}(w) \Gamma( \sqrt{\frac{pq}{t}} z^{\pm 1}x^{\pm 1}, st^{\frac14} \sqrt{\frac{pq}{t^{n-1}}} x^{\pm 1}y^{\pm 1}, \frac{\sqrt{pq}}{st^{\frac14}} x^{\pm 1}w^{\pm 1}, t^{\frac12(n-1)} y^{\pm 1}w^{\pm 1}) 
\\ & \times \Gamma(\sqrt{\frac{pq}{t^n}}v^{\pm 1}y^{\pm 1}) 
\prod_{r=1}^4 \Gamma( \frac{t_rt^{\frac14}}{s} z^{\pm 1}, su_rt^{\frac34} z^{\pm 1}, \frac{\sqrt{pq}st^{\frac14}}{t_r}x^{\pm 1}, \frac{\sqrt{pq}}{su_rt^{\frac14}} x^{\pm 1}, t_rw^{\pm 1}, b_rt^{\frac12n-\frac12} y^{\pm 1})
\\ &=
\prod_{1\leq r<s\leq 4} \Gamma(t_rt_s, t^{n-2}u_ru_s)
\prod_{k=1}^{n-2} \prod_{r,s=1}^4 \Gamma( u_rt_s t^{n-k-1})
%\\ & \qquad \times 
\iint \Delta_{I\!I}^{(n-1)}(t;z) \Delta_{I}^{(1)}(y) 
\\ & \qquad \times \Gamma( \sqrt{\frac{pq}{t}} y^{\pm 1}z^{\pm 1}, \sqrt{\frac{pq}{t^n}}v^{\pm 1}y^{\pm 1})
\prod_{r=1}^4 \Gamma( \frac{\sqrt{pq}}{u_r t^{\frac12n-1}} z^{\pm 1}, \frac{\sqrt{pq}}{t_r t^{\frac12n-1}} z^{\pm 1}, t_r t^{\frac12(n-1)} y^{\pm 1}, 
b_r t^{\frac12(n-1)} y^{\pm 1})
\end{align*}
\end{lemma}
\begin{proof}
We follow the proof from \eqref{eqVII}, through the proof of Lemma \ref{lem52} to \eqref{eqbbb}. That is, we first apply an $W(E_7)$-transformation in $w$, then apply an (inverse) induction enabler, and then use the $(n-1)$-dimensional version of the almost-Selberg transformation. It is not hard to check that applying Fubini is valid in all integrals.
\end{proof}
The next lemma relates \eqref{eqIII} to another induction enabler applied to \eqref{eqbbb}. Note that we can specialize this equation in $u_r\to \sqrt{t},\sqrt{qt},\sqrt{pt},\sqrt{pqt}$.
\begin{lemma}\label{lem73}
Under the balancing conditions $t^{2n-1}\prod_{r=1}^4 t_r=pq$ and $\prod_{r=1}^4 u_r=t^2pq$ we have
\begin{align*}
\iint & \Delta_{I\!I}^{(n-1)} (t;z) \hat \Delta_I^{(1)}(y) 
%\\ & \qquad \times 
\Gamma(-\sqrt{\frac{pq}{t}} y^{\pm 2} z^{\pm 1}) \prod_{r=1}^4 \Gamma( \frac{\sqrt{pq}}{t_rt^{n-1}}z^{\pm 1}, -t_rt^{n-\frac12} y^{\pm 2}) \hat \Gamma( \left( \frac{pq}{t^{2n}}\right)^{\frac14} v^{\pm 1}y^{\pm 1})
\prod_{r=1}^4 \Gamma(\frac{\sqrt{pq}t}{u_r} z^{\pm 1}) 
\\ &= 
\frac{\Gamma(t^n,t^{n-1})}{\hat \Gamma(t^n)} 
\prod_{1\leq r<s\leq 4} \Gamma(u_ru_s) \prod_{k=1}^{n-2} \prod_{r,s=1}^4 \Gamma(\frac{pq}{t_ru_s t^{n-k-1}})
\\& \qquad \times 
\iiint \Delta_{I\!I}^{(n-2)}(t;z) \Delta_I^{(1)}(y) \Delta_I^{(1)}(x) \Gamma(\sqrt{\frac{pq}{t}} z^{\pm 1}x^{\pm 1}, \sqrt{\frac{pq}{t^{n-1}}} x^{\pm 1}y^{\pm 1}, -\sqrt{\frac{pq}{t^n}} v^{\pm 2}y^{\pm 1}) 
\\ & \qquad \times \prod_{r=1}^4 \Gamma( \sqrt{t}u_r z^{\pm 1},\sqrt{t}t_r z^{\pm 1}, \frac{\sqrt{pq}}{t_r} x^{\pm 1}, \frac{\sqrt{pq}}{u_r} x^{\pm 1}, t_rt^{\frac12(n-1)} y^{\pm 1})
\hat \Gamma( t^{\frac12n} y^{\pm 1})
\end{align*}
\end{lemma}
\begin{proof}
We follow the main proof from \eqref{eqIII} to \eqref{eqVII}, subsequently apply the previous lemma specialized in $v\to -v^2$ and $b_r\to \sqrt{t}, \sqrt{pt}, \sqrt{qt}, \sqrt{pqt}$ (so $s\to t^{-\frac12n-\frac12}$), and finally apply an induction enabler. Note that Fubini is valid in all these equations, in particular both sides in the equation of the previous lemma pose no problems when specialized.
\end{proof}
Using this lemma, we can equate the desired right hand side of \eqref{eqthm71} to this doubly induction enabled version of the left hand side.
\begin{lemma}\label{lem74}
Under the balancing condition $t^{2n-1} \prod_{r=1}^4 t_r=pq$ we have
\begin{align*}
\prod_{i=0}^{n-1} & \prod_{1\leq r<s\leq 4} \Gamma(t^{2i} t_rt_s) \int \hat \Delta_{I\!I}^{(n)}(t;z) 
\prod_{r=1}^4 \Gamma(-t_r\sqrt{t} z^{\pm 2}) \hat \Gamma(\left( \frac{pq}{t^2}\right)^{\frac14} v^{\pm 1}z^{\pm 1})
\\ & =
\Gamma(t^n,t^{n-1}) 
\frac{ \hat \Gamma(t,\sqrt{pq}t)}{\Gamma(t,pqt)}
\prod_{1\leq r<s\leq 4} \Gamma(t_rt_s) 
\prod_{r=1}^4 \hat \Gamma(\sqrt{t} t_r)
\\ & \qquad \times 
\iiint \Delta_{I\!I}^{(n-2)}(t;z) \Delta_I^{(1)}(y) \Delta_I^{(1)}(x) \Gamma(\sqrt{\frac{pq}{t}} z^{\pm 1}x^{\pm 1}, \sqrt{\frac{pq}{t^{n-1}}} x^{\pm 1}y^{\pm 1}, -\sqrt{\frac{pq}{t^n}} v^{\pm 2}y^{\pm 1} ) 
\\ & \qquad \times \prod_{r=1}^4 \Gamma( \sqrt{t} t_rz^{\pm 1}, \frac{\sqrt{pq}}{t_r} x^{\pm 1}, t_rt^{\frac12(n-1)} y^{\pm 1}) 
\hat \Gamma(tz^{\pm 1}, \frac{1}{\sqrt{t}} x^{\pm 1}, t^{\frac12n} y^{\pm 1})
\end{align*}
\end{lemma}
\begin{proof}
Us the calculation as given above until \eqref{eqIII}. Subsequently apply Lemma \ref{lem73} specialized in $u_r\to \sqrt{t}, \sqrt{pt}, \sqrt{qt}, \sqrt{pqt}$.
\end{proof}
Unfortunately we can't apply the induction enabler transformation to this last equation, as it would lead to an evaluation of an integral at an apparent pole, so we need to perform two induction enablers at once, as described by the following lemma.
\begin{lemma}
Under the balancing condition $t^{2n-3} \prod_{r=1}^8 t_r=(pq)^2$ we have
\begin{align*}
\int & \Delta_{I\!I}^{(n)}(t;z) \prod_{r=1}^8 \Gamma(t_rz^{\pm 1}) \Gamma(\sqrt{\frac{pq}{t}} v^{\pm 1}z^{\pm 1}) 
\\ & = 
\Gamma(t^n,t^{n-1}) \prod_{1\leq r<s\leq 4} \Gamma(t_rt_s) 
 \iiint \Delta_{I\!I}^{(n-2)}(t;z) \Delta_{I}^{(1)}(y) \Delta_{I}^{(1)}(x) 
 \\ & \qquad \times 
\Gamma(\sqrt{\frac{pq}{t}} z^{\pm 1}x^{\pm 1}, \sqrt{\frac{pq}{t^{n-1}}} y^{\pm 1}x^{\pm1}, \sqrt{\frac{pq}{t^n}} v^{\pm 1}y^{\pm 1})
\prod_{r=1}^8 \Gamma(\sqrt{t} t_r z^{\pm 1}, \frac{\sqrt{pq}}{t_r} x^{\pm 1}, t_rt^{\frac12(n-1)} y^{\pm 1})
\end{align*}
\end{lemma}
\begin{proof}
This is the double iterate of the induction enabler transformation, Proposition \ref{propbluebox}.
\end{proof}
We can now finish the actual proof by starting with the right hand side of \eqref{eqthm71}, applying Lemma \ref{lem74} and finish by using the previous lemma specialized at $u_r\to \sqrt{t},\sqrt{pt},\sqrt{qt},\sqrt{pqt}$. 
\end{proof}

Finally we obtain the result from \cite[Conjecture Q1]{RainsLittlewood} in the case $\lambda=0$ as a corollary (to relate it closer to our theorem, we replaced $p$, $q$, and $t$ in the conjecture by their square roots).
\begin{cor}
Let $t_r\in \mathbb{C}^*$ satisfy the balancing condition $t^{n} \prod_{r=1}^4 t_r=-\sqrt{pq}$. Then we have
\[
\int \Delta_{I\!I}^{(n)}(t;z) \prod_{r=1}^4 \Gamma( t_r^2 z^{\pm 1}) \hat \Gamma(\sqrt{t} z^{\pm 1})
 = 
\prod_{i=0}^{2n-1} \prod_{1\leq r<s\leq 4} \hat \Gamma(-t^{\frac12 i} t_rt_s)
\int \hat \Delta_{I\!I}^{(n)}(t;z) \prod_{r=1}^4 \hat \Gamma( t_r z^{\pm 1}, \sqrt{t}t_r z^{\pm 1})
\]
\end{cor}
\begin{proof}
In Theorem \ref{thm71} we replace $t_r$ by $t_r^2$ for $r=1,2,3$, $t_4$ by $tt_4^2$, and $v$ by $i t_4 (t/pq)^{\frac14}$. This leads to the identity (where we simplified both sides by removing a pair of elliptic gamma functions using the reflection equation):
\begin{multline*}
\int \Delta_{I\!I}^{(n)}(t;z) 
\prod_{r=1}^4 \Gamma(t_r^2 z^{\pm 1})  
 \hat \Gamma(\sqrt{t} z^{\pm 1}) 
 \\ =
\prod_{i=0}^{n-1} \prod_{1\leq r<s\leq 3} \Gamma(t^{2i} t_r^2t_s^2)
\prod_{r=1}^3 \Gamma( t^{2i+1} t_r^2t_4^2)
\int \hat \Delta_{I\!I}^{(n)}(t;z)
 \prod_{r=1}^3 \Gamma(-t_r^2 \sqrt{t} z^{\pm 2}) 
\hat  \Gamma(- i t^{\frac34} t_4z^{\pm 1},i t_4 t^{-\frac14}  z^{\pm 1}) 
\end{multline*}
Now both sides of the equation are elliptic Selberg integrals with eight parameters, that is, we can use the Selberg transformation \eqref{eqtrafo2} on either side. In particular applying it on the right hand side with the group of four parameters being
$it_1t^{\frac14}$, $it_2 t^{\frac14}$, $it_3t^{\frac14}$, and $it_4^{-\frac14}$, gives us the desired equation.
\end{proof}

\end{document}